\documentclass[1p,reqno]{amsart}

\usepackage{lineno,hyperref}

\usepackage[a4paper, hmargin={2.7cm,2.7cm},vmargin={3.3cm,3.3cm}]{geometry}
\usepackage[T1]{fontenc}
\usepackage[utf8]{inputenc}

\usepackage{amsmath,amssymb,mathtools,amsthm,thmtools,mathrsfs}
\usepackage{tikz}
\usetikzlibrary{matrix}
\usepackage{cleveref}

\DeclareMathOperator{\spn}{span}
\newcommand{\dee}{\text{d}}
\DeclareMathOperator{\vol}{vol}
\DeclareMathOperator{\Vect}{Vect}
\newcommand{\an}{\Phi} 
\newcommand{\sy}{\Psi} 
\newcommand{\ft}{\Theta}
\newcommand{\N}{\mathbb{N}}   
\newcommand{\Z}{\mathbb{Z}}   
\newcommand{\Q}{\mathbb{Q}}   
\newcommand{\R}{\mathbb{R}}   
\newcommand{\C}{\mathbb{C}}   
\newcommand{\T}{\mathbb{T}}   
\newcommand{\id}{\text{id}}
\newcommand{\tr}{\text{tr}}
\newcommand{\lfb}{K}
\newcommand{\ufb}{L}
\DeclareMathOperator{\supp}{supp}
\def\lhs#1#2{{_\bullet\!\!}\left\langle #1,#2\right\rangle}

\DeclareMathOperator{\QP}{QP}
\def\sec#1#2{\QP_{#1,#2}}
\def\heis#1#2{\mathcal{E}_{#2}(#1)}
\def\sol#1{\mathcal{S}_{#1}}

\declaretheorem[style = plain, numberwithin = section]{theorem}
\newtheorem{thmx}{Theorem}

\declaretheorem[style = plain,      sibling = theorem]{corollary}
\declaretheorem[style = plain,      sibling = theorem]{lemma}
\declaretheorem[style = plain,      sibling = theorem]{proposition}
\declaretheorem[style = definition, sibling = theorem]{definition}
\declaretheorem[style = definition, sibling = theorem]{example}
\declaretheorem[style = remark,    numbered = no]{remark}
\declaretheorem[style = plain, sibling = theorem]{statement}

\RequirePackage[backend    = biber,
                sortcites  = true,
                giveninits = true,
                doi        = false,
                isbn       = false,
                url        = false,
                maxnames = 50,
                style      = numeric-comp]{biblatex}
\DeclareNameAlias{sortname}{family-given}
\DeclareNameAlias{default}{family-given}
\title{The Balian--Low theorem for locally compact abelian groups and vector bundles}

\author{Ulrik Enstad}
\address{Department of Mathematics,
Stockholm University,
SE-106 91 Stockholm, Sweden.}
\email{ulrik.enstad@math.su.se}

\begin{document}

\maketitle

\begin{abstract}
Let $\Lambda$ be a lattice in a second countable, locally compact abelian group $G$ with annihilator $\Lambda^{\perp} \subseteq \widehat{G}$. We investigate the validity of the following statement: For every $\eta$ in the Feichtinger algebra $S_0(G)$, the Gabor system $\{ M_{\tau} T_{\lambda} \eta \}_{\lambda \in \Lambda, \tau \in \Lambda^{\perp}}$ is not a frame for $L^2(G)$. When $\Lambda$ is a lattice in $G = \R$, this statement is a variant of the Balian--Low theorem. Extending a result of R.\ Balan, we show that whether the statement generalizes to $(G,\Lambda)$ is equivalent to the nontriviality of a certain vector bundle over the compact space $(G/\Lambda) \times (\widehat{G}/\Lambda^{\perp})$. We prove this equivalence using Heisenberg modules. More specifically, we show that the Zak transform can be viewed as an isomorphism of certain Hilbert $C^*$-modules. As an application, we prove a Balian--Low theorem in the new context of the group $\R \times \Q_p$, where $\Q_p$ denotes the $p$-adic numbers.
\end{abstract}

\section{Introduction}

\noindent
In his work on projective modules over noncommutative tori \cite{rieffel88}, M.\ Rieffel introduced a class of Hilbert $C^*$-modules known as Heisenberg modules. These modules establish the Morita equivalence of twisted group $C^*$-algebras associated to a lattice $\Delta$ in the time-frequency plane of a locally compact abelian (LCA) group $G$. Heisenberg modules have been applied numerous times in operator algebras and noncommutative geometry, see for example \cite{sol1,sol2,ChLu19,Bo99,LeMo16}.

It was shown by F.\ Luef in \cite{projmod,luefprojections} that the Morita equivalence of Heisenberg modules over noncommutative tori is closely related to the duality theory of regular Gabor frames. These connections were recently generalized to the setting of LCA groups \cite{gaborheis}. Gabor frames are the objects of study in Gabor analysis, which can be considered a subfield of time-frequency analysis. The central problem of Gabor analysis is the recovery of signals from a discrete set of time-frequency translates of fixed functions in $L^2(\R^n)$. While Gabor analysis is usually carried out in $\R^n$, much of the framework can be generalized to the setting of a LCA group $G$ as follows: If $(x,\omega)$ is an element of the time-frequency plane $G \times \widehat{G}$, the time-frequency shift operator $\pi(x,\omega)$ acts on functions $\xi \in L^2(G)$ via
\[ \pi(x,\omega)\xi(t) = \omega(t) \xi(x^{-1}t) \]
for $t \in G$. By picking a lattice $\Delta$ in the time-frequency plane $G \times \widehat{G}$ and a finite set of functions $\eta_1, \ldots, \eta_k \in L^2(G)$, one forms the associated \emph{multiwindow Gabor system} as follows:
\[ \mathcal{G}(\eta_1,\ldots,\eta_k;\Delta) = \{ \pi(z) \eta_j : z \in \Delta, 1 \leq j \leq k \} .\]
To allow for stable reconstruction of functions in $L^2(G)$ from a multiwindow Gabor system, one requires the frame property due to Duffin and Schaeffer \cite{DuSc52} to be satisfied. That is, if there exist constants $\lfb, \ufb > 0$ such that
\[ \lfb \| \xi \|_2^2 \leq \sum_{j=1}^k \sum_{z \in \Delta} | \langle \xi, \pi(z)\eta_j \rangle |^2 \leq \ufb \| \xi \|_2^2 \]
for all $\xi \in L^2(G)$, one calls the multiwindow Gabor system $\mathcal{G}(\eta_1, \ldots, \eta_k ; \Delta)$ a \emph{multiwindow Gabor frame}. In particular, a singlewindow Gabor system, or just Gabor system for short, is called a \emph{Gabor frame} if it forms a frame for $L^2(G)$.

One of the main observations of \cite{projmod,gaborheis} is that if the windows $\eta_1, \ldots, \eta_k$ of a multiwindow Gabor frame over the lattice $\Delta \subseteq G \times \widehat{G}$ are well-localized, they can be interpreted as a set of generators for the Heisenberg module $\heis{G}{\Delta}$ constructed from $\Delta$. By well-localized, we mean that the generators are all elements of the Feichtinger algebra $S_0(G)$, a Banach space of test functions that is fundamental to time-frequency analysis. Since Heisenberg modules are finitely generated, an immediate consequence is the existence of a multiwindow Gabor frame $\mathcal{G}(\eta_1,\ldots,\eta_k;\Delta)$ with $\eta_j \in S_0(G)$, $1 \leq j \leq k$, for any given lattice $\Delta$ in $G \times \widehat{G}$.

A longstanding problem in Gabor analysis on $G = \R^n$ is whether one can find a (singlewindow) Gabor frame $\mathcal{G}(\eta;\Delta)$ with a well-localized window $\eta \in S_0(\R^n)$ over a given lattice $\Delta$ in the time-frequency plane $\R^n \times \widehat{\R}^n \cong \R^{2n}$. One can ask the same question for a lattice $\Delta$ in the time-frequency plane $G \times \widehat{G}$ of a locally compact abelian group $G$: Does there exist an $\eta \in S_0(G)$ for which $\mathcal{G}(\eta,\Delta)$ is a Gabor frame? Or, in terms of Heisenberg modules, is the Heisenberg module $\heis{G}{\Delta}$ singly generated? A basic restriction on the lattice is provided by one of the density theorems \cite{density_hist}: It is necessary that $\vol(\Delta) \leq 1$ for Gabor frames $\mathcal{G}(\eta,\Delta)$ with $\eta \in S_0(G)$ to exist \cite[Theorem 5.6]{jakobsen_density}. But in the group $G = \R^n$, more is true:

\begin{theorem}\label{thm:bal_lattice}
Let $\Delta$ be a lattice in the time-frequency plane $\R^n \times \widehat{\R^n}$ of $\R^n$. If there exists a function $\eta$ in the Feichtinger algebra $S_0(\R^n)$ for which $\mathcal{G}(\eta,\Delta)$ is a Gabor frame for $L^2(\R^n)$, then $\vol(\Delta) < 1$.
\end{theorem}

The above result is an example of a Balian-Low theorem (BLT), as it is a non-existence result for well-localized Gabor frames at the so-called \emph{critical density} $\vol(\Delta)=1$. The original Balian-Low theorem is due to R.\ Balian \cite{balian} and F.\ Low \cite{low} and concerns lattices of the form $\Delta = \alpha \Z \times \beta \Z$ for $\alpha,\beta > 0$ in $\R \times \widehat{\R}$. It also uses a slightly more general notion of time-frequency localization that does not involve the Feichtinger algebra, but which is particular to $G = \R$. In \cite{battle}, a proof of the original formulation of the Balian--Low theorem is deduced from the uncertainty principle, and a relation to noncommutative geometry is demonstrated in \cite{balian_ncg}. The amalgam Balian--Low theorem is another early version of the BLT that employs Wiener amalgam spaces \cite{benheilnut,heilwiener}. Versions of the Balian--Low thorem for more general lattices in $\R^n$ have since been proved, even for discrete sets $\Delta$ without any lattice structure \cite{AsFeKai14}. \Cref{thm:bal_lattice} can also be deduced from the pertubation results of H.\ Feichtinger and N.\ Kaiblinger \cite{FeKa04}. The converse of \Cref{thm:bal_lattice} remains open for general lattices, but was proved for the class of non-rational lattices in \cite{gaborheis}. The proof uses Heisenberg modules and $K$-theory of noncommutative tori.

It is no coincidence that the setting of \Cref{thm:bal_lattice} is the group $G = \R^n$. Indeed, \Cref{thm:bal_lattice} is easily shown to fail if one replaces $\R^n$ with an arbitrary LCA group $G$ \cite{groloc}. For instance, it fails for discrete or compact groups (see \Cref{prop:compact_or_discrete}). One might then ask if there is a way to characterize the groups $G$ for which \Cref{thm:bal_lattice} holds. This will be the setup in the present paper, but we restrict ourselves to the case where the lattice $\Delta$ takes the form $\Lambda \times \Lambda^{\perp}$, where $\Lambda$ is a lattice in $G$ and $\Lambda^{\perp}$ is the annihilator of $\Lambda$ in $\widehat{G}$ \eqref{eq:annihilator}. These lattices always have volume 1. We will consider the following statement in the setting $(G,\Lambda)$, where $G$ is a second countable LCA group and $\Lambda$ is a lattice in $G$.

\begin{statement}\label{problem:bl}
For all $\eta \in S_0(G)$, the Gabor system
\[ \mathcal{G}(\eta, \Lambda \times \Lambda^{\perp}) = \{ \pi(\lambda,\tau) \eta : \lambda \in \Lambda, \tau \in \Lambda^{\perp} \} \]
is not a frame for $L^2(G)$.
\end{statement}

Note that it is a consequence of \Cref{thm:bal_lattice} that \Cref{problem:bl} holds true for any lattice $\Lambda$ in $G=\R^n$.

The advantage of the formulation of \Cref{problem:bl} is that the \emph{Zak transform} can be employed. Originally introduced by I.\ Gelfand \cite{gelfand}, the Zak transform was later generalized by A.\ Weil to the case of locally compact abelian groups \cite{weil}. It takes it name from the physicist J.\ Zak who discovered it independently  \cite{zak1}. Given a lattice $\Lambda$ in the LCA group $G$, the Zak transform of a complex-valued function $\xi$ on $G$ is the function $Z_{G,\Lambda} \xi \colon G \times \widehat{G} \to \C$ given by
\[ Z_{G,\Lambda}\xi(x,\omega) = \sum_{\lambda \in \Lambda} \xi(x\lambda)\omega(\lambda) \]
for $(x,\omega) \in G \times \widehat{G}$. The above defines a continuous function if e.g.\ $\xi \in S_0(G)$.

The first proofs of the amalgam version of the BLT \cite{benheilnut,heilwiener} employed the Zak transform, and used the fact that it diagonalizes the frame operator associated to the Gabor system $\mathcal{G}(\eta, \Lambda \times \Lambda^{\perp})$. In \cite{kankut}, E.\ Kaniuth and G.\ Kutyniok used the Zak transform to show that the Balian--Low statement in \Cref{problem:bl} holds for all lattices in compactly generated, second countable, locally compact abelian groups with noncompact identity component. It is an open problem whether the hypothesis that $G$ is compactly generated can be dropped from their theorem.

One of the main points we want to make in this paper is that the Zak transform has a natural interpretation as an isomorphism of Hilbert $C^*$-modules. When $\Delta = \Lambda \times \Lambda^{\perp}$, the associated Heisenberg module $\heis{G}{\Delta}$ becomes a Hilbert $C^*$-module over the (un-twisted) group $C^*$-algebra of $\Delta$. Using the Fourier transform, this algebra can be identified as the continuous functions on the Pontryagin dual $X \coloneqq \widehat{\Delta} \cong (G/\Lambda) \times (\widehat{G}/\Lambda^{\perp})$. By the Serre--Swan theorem, the projective module $\heis{G}{\Delta}$ must be isomorphic to a vector bundle over $X$. The role of the Zak transform in this respect is to identify the vector bundle in question. More precisely, we prove that there exists a complex vector bundle $E_{G,\Lambda}$ for which the following holds:

\begin{thmx}[cf.\ \Cref{thm:module_iso} / \Cref{prop:hilbert_sec}] \label{thm:intro2}
Let $\Lambda$ be a lattice in the second countable LCA group $G$, and let $\Delta$ be the lattice $\Lambda \times \Lambda^{\perp}$ in $G \times \widehat{G}$. Then the Zak transform implements an isomorphism of Hilbert $C^*$-modules
\[ Z_{G,\Lambda} : \heis{G}{\Delta} \to \Gamma(E_{G,\Lambda}). \]
Here, $\heis{G}{\Delta}$ is the Heisenberg module associated to $\Delta$, and $\Gamma(E_{G,\Lambda})$ is the module of continuous sections of the complex line bundle $E_{G,\Lambda}$ constructed in \Cref{sec:bundle}.
\end{thmx}

Now \Cref{problem:bl} is equivalent to the nonexistence of a single generator for the Heisenberg module $\heis{G}{\Delta}$. We show in \Cref{sec:modules_sections} that the $C(X)$-module $\Gamma(E)$ of continuous sections of a vector bundle $E$ over a compact Hausdorff space is singly generated if and only if $E$ is a trivial bundle. Consequently, we can formulate \Cref{problem:bl} in terms of the vector bundle $E_{G,\Lambda}$ from \Cref{thm:intro2} as follows:

\begin{thmx}[cf.\ \Cref{thm:bundle_bal}]\label{thm:intro1}
Let $\Lambda$ be a lattice in a second countable, locally compact abelian group $G$. Then the following are equivalent:
\begin{enumerate}
\item $E_{G,\Lambda}$ is nontrivial.
\item The Balian--Low statement (\Cref{problem:bl}) holds in this setting. That is, whenever $\eta \in S_0(G)$, then the Gabor system
\[ \mathcal{G}(\eta,\Lambda \times \Lambda^{\perp}) = \{ \pi(\lambda,\tau) \eta : \lambda \in \Lambda, \tau \in \Lambda^{\perp} \} \]
is not a frame for $L^2(G)$.
\end{enumerate}
\end{thmx}

The above theorem builds upon an idea of R. Balan connecting Gabor superframes to vector bundles over the 2-torus $\T^2$ \cite{balan}. A special case of his result is that the amalgam version of the Balian--Low theorem is a consequence of the nontriviality of a certain line bundle over $\T^2$. If $G = \R$ and $\Lambda = \alpha \Z$ in \Cref{thm:intro1}, then the base space $X_{G,\Lambda}$ is homeomorphic to $\T^2$, and we will indeed show that $E_{G,\Lambda}$ in this case is closely related to Balan's bundle (cf.\ \Cref{ex:balanbundle}). Thus, \Cref{thm:intro1} can be viewed as an extension of this special case of Balan's result to general second countable LCA groups.

We end this paper by applying \Cref{thm:intro1} to prove that the Balian--Low statement (\Cref{problem:bl}) holds in a new setting: We set $G$ to be the truncated adele group $\R \times \Q_p$ where $\Q_p$ denotes the $p$-adic numbers. It is a well-known fact from number theory that the group of $p$-adic rationals $\Lambda = \Z[1/p] = \{ a/p^k : a,k\in \Z \}$ embeds as a lattice in $G$. We then have the following:

\begin{thmx}[cf.\ \Cref{thm:rqp}]\label{thm:intro3}
Let $G$ be the group $\R \times \Q_p$, and let $\Lambda$ be the lattice $\Z[1/p]$ embedded into $G$ as in \Cref{sec:rp}. Then The Balian--Low statement (\Cref{problem:bl}) holds for $(G,\Lambda)$: That is, whenever $\eta \in S_0(\R \times \Q_p)$, the Gabor system
\[ \mathcal{G}(\eta, \Lambda \times \Lambda^{\perp}) = \{ (s,x) \mapsto e^{2\pi i rs} e^{-2\pi i \{ r x \}_p} \eta( s - q, x - q) : q,r \in \Z[1/p] \} \]
is not a frame for $L^2(\R \times \Q_p)$.
\end{thmx}

The above theorem is the first Balian--Low theorem in the context of a LCA group which is not compactly generated. Hence, it is not covered by the result of Kaniuth and Kutyniok \cite{kankut}. Number-theoretic groups such as $\R \times \Q_p$ and the full adeles over the rationals have not been explored much in Gabor analysis so far. In \cite{adeles}, examples of Gabor frames in these groups are constructed, and a mild Balian--Low type theorem, namely \cite[Proposition 4.4]{adeles}, is also proved. However, the result only holds for functions in $S_0(\R \times \Q_p)$ of a very specific form, and \Cref{thm:intro3} is a generalization to all generators in $S_0(\R \times \Q_p)$. \\

The text is structured as follows: In \Cref{sec:hilbertmodules}, we define Hilbert $C^*$-modules and their frames, and we discuss modules of sections of vector bundles. In \Cref{sec:tfa}, we introduce Gabor analysis on locally compact abelian groups and Heisenberg modules. In \Cref{sec:bundles_qp}, we introduce the Zak transform, quasiperiodic functions and the vector bundle $E_{G,\Lambda}$ from \Cref{thm:intro2}. In \Cref{sec:connecting}, we show that the Zak transform gives an isomorphism of Hilbert $C^*$-modules, and prove theorems \Cref{thm:intro2} and \Cref{thm:intro1}. Then in \Cref{sec:rp}, we prove a Balian--Low theorem for the group $\R \times \Q_p$, namely \Cref{thm:intro3}. In the appendix, we have collected some basic results that are needed but do not constitute a part of the spirit of the main text.

\section{Hilbert \texorpdfstring{$C^*$}{C*}-modules and their frames}\label{sec:hilbertmodules}

\subsection{Hilbert \texorpdfstring{$C^*$}{C*}-modules}

In this section we define Hilbert $C^*$-modules, see e.g.\ \cite{raeburns}. Let $A$ be a unital $C^*$-algebra. A Hilbert $C^*$-module over $A$ is, roughly speaking, a ``Hilbert space'' with an inner product taking values in the $C^*$-algebra $A$ rather than the complex numbers. More precisely, a \emph{left Hilbert $A$-module} $\mathcal{E}$ is a left $A$-module that is equipped with an $A$-valued inner product $\lhs{ \cdot}{\cdot} \colon \mathcal{E} \times \mathcal{E} \to A$ that satisfies the following axioms:
\begin{enumerate}
\item $\lhs{a \xi + b \eta}{\gamma} = a\lhs{\xi}{\gamma} + b \lhs{\eta}{\gamma}$ for $a,b \in A$ and $\xi, \eta, \gamma \in \mathcal{E}$.
\item $\lhs{\xi}{\eta}^* = \lhs{\eta}{\xi}$ for $\xi,\eta \in \mathcal{E}$.
\item $\lhs{\xi}{\xi} \geq 0$, with $\lhs{\xi}{\xi} = 0$ if and only if $\xi = 0$.
\item $\mathcal{E}$ is complete with respect to the norm $\| \xi \|_{\mathcal{E}} = \| \lhs{\xi}{\xi} \|^{1/2}$.
\end{enumerate}
Note that we use the dot in the notation for the $A$-valued inner product to distinguish it from $\C$-valued inner products.


From the definition above, $\mathcal{E}$ has a complex vector space structure given by $ \lambda \xi \coloneqq (\lambda 1_A) \xi$ where $\lambda \in \C$, $\xi \in \mathcal{E}$ and $1_A$ is the multiplicative identity of $A$.

If $A = \C$, one recovers the definition of a (complex) Hilbert space. \\

Let $\mathcal{E}$ and $\mathcal{F}$ be left Hilbert $A$-modules, and let $T \colon \mathcal{E} \to \mathcal{F}$ be a bounded, $A$-linear map. Then $T$ is called \emph{adjointable} if there exists an $A$-linear map $S \colon \mathcal{F} \to \mathcal{E}$ such that
\[ \lhs{ T \xi}{\eta} = \lhs{\xi}{S \eta} \]
for all $\xi \in \mathcal{E}$ and $\eta \in \mathcal{F}$. The map $S$ is then uniquely determined and is referred to the \emph{adjoint} of $S$. It is denoted by $S = T^*$. If $T$ is adjointable, then $T$ is invertible if and only if its adjoint is invertible, and in that case we have the relationship $(T^*)^{-1} = (T^{-1})^*$. The set of adjointable operators $\mathcal{E} \to \mathcal{E}$ is a $C^*$-algebra and is denoted by $\mathcal{L}_A(\mathcal{E})$ or just $\mathcal{L}(\mathcal{E})$.

A bounded, $A$-linear map $T \colon \mathcal{E} \to \mathcal{F}$ is called \emph{inner product preserving} if
\[ \lhs{T \xi}{T \eta} = \lhs{\xi}{\eta} \]
for all $\xi, \eta \in \mathcal{E}$. An inner product preserving map is automatically injective, as it becomes an isometry in the sense of metric spaces. An inner product preserving map is not necessarily adjointable, see \cite[Example 2.19]{raeburns}. However, if it is surjective, hence invertible, then it is also adjointable, and the adjoint is given by the inverse.

We say that two left Hilbert $A$-modules $\mathcal{E}$ and $\mathcal{F}$ are \emph{isomorphic} if there exists a surjective inner product preserving (hence invertible and adjointable) $A$-linear map $T \colon \mathcal{E} \to \mathcal{F}$.\\


Let $J$ be a countable index set. Denote by $\ell^2(J,A)$ the set of all sequences $(a_j)_{j \in J}$ for which the sum $\sum_{j \in J} a_j a_j^*$ converges unconditionally in $A$ (this is the same as the direct sum $\bigoplus_{j \in J} A$, see \cite[p.\ 6]{La95}). This set forms a left Hilbert $A$-module with respect to pointwise left multiplication and the inner product
\[ \lhs{(a_j)_j}{(b_j)_j} = \sum_{j \in J} a_j b_j^* \]
for $(a_j)_j, (b_j)_j \in \ell^2(J,A)$. If $J = \{ 1, \ldots, k\}$ for some $k \in \N$, one obtains the left Hilbert $A$-module $A^k$, the \emph{free} $A$-module of rank $k$. We represent the elements of $A^k$ as row vectors. If $M$ is a $k \times k$ matrix with entries in $A$, then \emph{right multiplication} by $M$ defines an $A$-linear map $A^n \to A^n$.

A left Hilbert $A$-module $\mathcal{E}$ is called \emph{countably generated} if there exists a countable set $S \subseteq \mathcal{E}$ such that $\{ a \cdot \xi : a \in A, \xi \in S \}$ is dense in $\mathcal{E}$. If $A = \C$ so that $\mathcal{E}$ is a Hilbert space, then being countably generated as an $A$-module is the same as being separable as a Hilbert space. The module $\mathcal{E}$ is \emph{finitely generated} if the set $S$ is finite. This is equivalent to the existence of a finite set $\{ \eta_1, \ldots, \eta_k \}$ in $\mathcal{E}$ with the following property: For every $\xi \in \mathcal{E}$, there exist elements $a_1, \ldots, a_k \in A$ such that
\[ \xi = a_1 \eta_1 + \cdots + a_k \eta_k .\]
The set $\{ \eta_1, \ldots, \eta_k \}$ is called a generating set for $\mathcal{E}$. Moreover, $\mathcal{E}$ is \emph{finitely generated projective} if is isomorphic to a module of the form $A^k P$, where $P$ is a $k \times k$ projection matrix with coefficients in $A$.

\subsection{Module frames}

A \emph{module frame} \cite{frank} in a left Hilbert $A$-module $\mathcal{E}$ is a sequence $(\eta_j)_{j \in J}$ in $\mathcal{E}$ for which there exists $\lfb,\ufb > 0$ such that the following double inequality holds for all $\xi \in \mathcal{E}$:
\begin{equation}
\lfb \lhs{\xi}{\xi} \leq \sum_{j \in J} \lhs{\xi}{\eta_j}\lhs{\xi}{\eta_j}^* \leq \ufb \lhs{\xi}{\xi} . \label{eq:frame_def}
\end{equation}
Here, the inequalities are to be interpreted with respect to the order structure on the self-adjoint elements of the $C^*$-algebra $A$.

Note that if $A$ is the $C^*$-algebra of complex numbers, then $\mathcal{E}$ is a Hilbert space and the above double inequality reduces to
\begin{equation}
\lfb \| \xi \|^2 \leq \sum_{j \in J} | \langle \xi, \eta_j \rangle |^2 \leq \ufb \| \xi \|^2
\end{equation}
for all $\xi \in \mathcal{H}$.

We call the numbers $\lfb$ and $\ufb$ \emph{lower} and \emph{upper frame bounds}, respectively. A frame is called \emph{tight} if one can choose frame bounds $\lfb=\ufb$, and \emph{normalized tight} if one can choose $\lfb=\ufb=1$.

A sequence $(\eta_j)_{j \in J}$ that satisfies the upper frame bound condition but not necessarily the lower, is called a \emph{Bessel sequence}. An upper frame bound for a Bessel sequence is called a \emph{Bessel bound} for the sequence. If $(\eta_j)_j$ is a Bessel sequence in $\mathcal{E}$, one defines the operators $\an \colon \mathcal{E} \to \ell^2(J,A)$ and $\sy \colon \ell^2(J,A) \to \mathcal{E}$ by
\begin{align}
    \an \xi &= ( \lhs{\xi}{\eta_j})_{j \in J} \\
    \sy (a_j)_j &= \sum_{j \in J} a_j \eta_j
\end{align}
for $\xi \in \mathcal{E}$ and $(a_j)_j \in \ell^2(J,A)$. These are adjointable $A$-linear operators, and $\an^* = \sy$. They go under different names in the literature. In \cite{frank}, $\an$ is called the \emph{frame transform}, while in Gabor analysis in general, they are called the \emph{analysis} and \emph{synthesis} operator, respectively.

One also defines $\ft = \sy \an$ and calls it the \emph{frame operator} corresponding to $(\eta_j)_j$. It is given by
\begin{equation}
    \ft \xi = \sum_{j \in J} \lhs{\xi}{\eta_j} \eta_j \label{eq:module_ft}
\end{equation}
for $\xi \in \mathcal{E}$.

The following gives an important characterization of the frame property of a Bessel sequence. It can be seen by noting that $\lhs{ \ft \xi}{\xi} = \sum_{j \in J} \lhs{\xi}{\eta_j}\lhs{\xi}{\eta_j}^*$ and using the fact that $T_1 \leq T_2$ if and only if $\lhs{T_1 \xi}{\xi} \leq \lhs{T_2 \xi}{\xi}$ for all $\xi \in \mathcal{E}$, where $T_1,T_2 \in \mathcal{L}(\mathcal{E})$ are self-adjoint \cite[Lemma 2.28]{raeburns}.

\begin{proposition}\label{prop:frame_char}
Let $\mathcal{E}$ be a left Hilbert $A$-module over a $C^*$-algebra $A$. Let $(\eta_j)_{j \in J}$ be a Bessel sequence in $\mathcal{E}$. Then the following are equivalent:
\begin{enumerate}
    \item $(\eta_j)_j$ is a frame for $\mathcal{E}$.
    \item The frame operator $\Theta$ corresponding to $(\eta_j)_j$ as in \eqref{eq:module_ft} is invertible in $\mathcal{L}(\mathcal{E})$.
\end{enumerate}
\end{proposition}

Using the Cauchy--Schwarz inequality for Hilbert $C^*$-modules, one can show that for a finite sequence, there is always an upper frame bound. Thus, to show that a finite sequence is a module frame, one only needs to show that there exists a lower frame bound. In fact, we have the following characterizations of finite module frames.

\begin{proposition}\label{prop:fingen_module}
Let $A$ be a unital $C^*$-algebra, and let $\mathcal{E}$ be a left Hilbert $A$-module. Let $\eta_1, \ldots, \eta_k \in \mathcal{E}$. Then the following are equivalent:
\begin{enumerate}
    \item The set $\{ \eta_1, \ldots, \eta_k \}$ is a generating set for $\mathcal{E}$.
    \item The sequence $(\eta_1, \ldots, \eta_k)$ is a module frame for $\mathcal{E}$.
    \item The frame operator $\Theta$ corresponding to $(\eta_1, \ldots, \eta_k)$ as in \eqref{eq:module_ft} is invertible in $\mathcal{L}(\mathcal{E})$.
\end{enumerate}
Moreover, if $(\eta_1, \ldots, \eta_k)$ is a module frame, then the sequence $( \tilde{\eta_1}, \ldots, \tilde{\eta_k})$ where $\tilde{\eta_j} = \ft^{-1/2} \eta_j$ for $1 \leq j \leq k$ is a normalized tight module frame for $\mathcal{E}$. Hence, $\mathcal{E}$ being finitely generated is equivalent to the existence of a normalized tight module frame in $\mathcal{E}$.
\end{proposition}

\begin{proof}
The equivalence of $(i)$ and $(ii)$ is proved in e.g.\ \cite[Proposition 5.9]{frank}. The equivalence of $(ii)$ and $(iii)$ is \Cref{prop:frame_char}. The last observation follows from writing out the identity
\[ \lhs{\xi}{\xi} = \lhs{\Theta^{-1/2} \Theta \Theta^{-1/2} \xi}{\xi} \]
for $\xi \in \mathcal{E}$, where $\Theta$ is the frame operator corresponding to the sequence $(\eta_j)_j$.
\end{proof}

An important application of normalized tight frames is the following, see \cite[Proposition 7.2]{Ri10} for a reference.

\begin{proposition}\label{prop:frame_proj}
Let $A$ be a unital $C^*$-algebra. Then every finitely generated left Hilbert $A$-module is projective. In fact, suppose $\{ \eta_1, \ldots, \eta_k \}$ is a normalized tight frame for the Hilbert $A$-module $\mathcal{E}$ (which exists by \Cref{prop:fingen_module}). Then the $k \times k$ matrix
\[ P = \begin{pmatrix} \lhs{ \eta_1}{\eta_1} & \cdots & \lhs{\eta_1}{\eta_k}  \\
 \vdots & \ddots & \vdots \\
 \lhs{\eta_k}{\eta_1} & \cdots & \lhs{\eta_k}{\eta_k} \end{pmatrix} \]
is a projection in $M_k(A)$, and $\mathcal{E} \cong A^k P$.
\end{proposition}

Note especially the first sentence of \Cref{prop:frame_proj}, namely that all finitely generated Hilbert $C^*$-modules are projective. One can show that this also holds more generally for countably generated projective $C^*$-modules, but this result will not be important to us.

\subsection{Modules of sections of a vector bundle}\label{sec:modules_sections}

In this section we look at finitely generated Hilbert $C^*$-modules over commutative $C^*$-algebras. If $A$ is commutative, then by Gelfand duality, $A$ is $*$-isomorphic to the $C^*$-algebra $C(X)$ of continuous functions on a uniquely determined compact Hausdorff space $X$.

A way to obtain Hilbert $C^*$-modules over $C(X)$ is from Hermitian vector bundles over $X$. If $\pi \colon E \to X$ is a Hermitian vector bundle, we denote by $\Gamma(E)$ the set of continuous sections of $E$, i.e.\ all continuous functions $s \colon X \to E$ such that $\pi \circ s = \id_X$. The set $\Gamma(E)$ has the structure of a finitely generated left Hilbert $C(X)$-module with the following left action and $C(X)$-valued inner product:
\begin{align}
(f \cdot s)(x) &= f(x) s(x) \label{eq:bundle1}, \\
\lhs{s}{t}(x) &= \langle s(x), t(x) \rangle_x. \label{eq:bundle2}
\end{align}
Here, $f \in C(X)$, $s,t \in \Gamma(E)$, $x \in X$ and $\langle \cdot, \cdot \rangle_x$ denotes the inner product on the fiber $E_x$ of $E$. We also use the notation $\| v \|_x^2 = \langle v, v \rangle_x$.

That $\Gamma(E)$ is finitely generated comes from the fact that for every vector bundle $E$ over a compact Hausdorff space $X$, there exists a vector bundle $F$ over $X$ such that $E \oplus F$ is isomorphic to a trivial bundle \cite[Corollary 5]{swan}.

The remarkable fact about finitely generated projective modules over $C(X)$ is that they all come from vector bundles over $X$. This is known as the Serre--Swan theorem. The version that we state here is slightly modified to include the inner product structures found on Hermitian vector bundles and Hilbert $C^*$-modules.

\begin{proposition}[Serre--Swan theorem]\label{prop:serre_swan}
Suppose $\mathcal{E}$ is a finitely generated left Hilbert $C(X)$-module. Then there exists a unique Hermitian vector bundle $E \to X$ such that $\Gamma(E) \cong \mathcal{E}$ as left Hilbert $C(X)$-modules. If $\mathcal{E}$ is represented as $C(X)^k P$ for a projection $P$ in $M_k(C(X))$, then $E$ can be constructed as the subbundle
\[ E = \{ (x,v) \in X \times \C^k : vP(x) = v \} \]
of the trivial bundle of rank $k$.
\end{proposition}

\begin{proof}
Note that by \Cref{prop:frame_proj}, $\mathcal{E}$ is projective. By the Serre--Swan theorem for finitely generated projective $C(X)$-modules \cite{swan}, there exists a unique complex vector bundle $E \to X$ such that $\Gamma(E) \cong \mathcal{E}$ as left $C(X)$-modules. The description of $E$ as a subbundle of $X \times \C^k$ also follows from \cite{swan}. We suppress the isomorphism from the notation and think of $\mathcal{E}$ as being equal to $\Gamma(E)$.

We define an inner product on the fibers of $E$ as follows: Let $x \in X$, $v,w \in E_x$, and pick sections $s,t \in \Gamma(E)$ such that $s(x) = v$ and $t(x) = w$. Set $\langle v, w \rangle_x = \lhs{s}{t}(x)$. This is independent of the choice of sections, and will give us the required inner product on $E$.
\end{proof}

The next proposition describes generating sets of $\Gamma(E)$ as a left Hilbert $C(X)$-module, where $E \to X$ is a Hermitian vector bundle. Recall that a line bundle is a vector bundle where the dimension of the fibers is equal to 1.

\begin{proposition}\label{prop:frame_sections_char}
Let $E$ be a Hermitian vector bundle over a compact Hausdorff space $X$. Let $s_1, \ldots, s_k \in \Gamma(E)$. The following are equivalent:
\begin{enumerate}
\item The set $\{s_1, \ldots, s_k \}$ is a generating set for $\Gamma(E)$ as a left $C(X)$-module.
\item For every $x \in X$, we have that
\[ \spn \{ s_1(x), \ldots, s_k(x) \} = E_x .\]
\end{enumerate}
\end{proposition}

\begin{proof}
The frame operator corresponding to the Bessel sequence $(s_j)_{j=1}^k$ in $\Gamma(E)$ is given by
\begin{equation}
\ft s = \sum_{j=1}^k \lhs{s}{s_j}s_j \label{eq:ft_sections}
\end{equation}
for $s \in \Gamma(E)$. This is a $C(X)$-linear operator $\Gamma(E) \to \Gamma(E)$, so by \cite[Theorem 1]{swan} there exists a unique vector bundle homomorphism $\theta : E \to E$ such that
\[ (\ft s)(x) = \theta_x(s(x)) \]
for all $s \in \Gamma(E)$ and $x \in X$, where $\theta_x$ denotes the restriction of $\theta$ to $E_x$. For each $x \in X$, let $\Theta_x : E_x \to E_x$ denote the frame operator of the sequence $\{ s_1(x), \ldots, s_k(x) \}$ in the Hilbert space $E_x$. By \eqref{eq:ft_sections}, we have that
\begin{equation}
\theta_x(s(x)) = \sum_{j=1}^k \lhs{s}{s_j} s_j(x) = \sum_{j=1}^k \langle s(x), s_j(x) \rangle_x s_j(x) = \Theta_x (s(x)) \label{eq:bundle_operators}
\end{equation}
for all $s \in \Gamma(E)$ and $x \in X$. Given $v \in E_x$, then letting $s \in \Gamma(E)$ be such that $s(x) = v$, we obtain that $\theta_x(v) = \Theta_x(v)$. Hence, by \eqref{eq:bundle_operators}, $\theta_x = \Theta_x$ for all $x \in X$.

By \Cref{prop:fingen_module}, $\{ s_1, \ldots, s_k \}$ is a generating set for $\Gamma(E)$ if and only if the frame operator $\Theta$ is invertible in $\Gamma(E)$. But this happens if and only if $\theta$ is invertible as a map of vector bundles. This is equivalent to each $\theta_x$ being invertible for every $x \in X$ by \cite[Lemma 2.3]{milnor}. This is equivalent to $\{ s_1(x), \ldots, s_k(x) \}$ being a generating set for $E_x$ for each $x \in X$, using \Cref{prop:fingen_module}. This finishes the proof.
\end{proof}

We immediately obtain the following corollary:

\begin{corollary}\label{cor:singlegen_bundle}
Suppose $E$ is a line bundle, and that $s_1, \ldots, s_k \in \Gamma(E)$. Then $\{ s_1, \ldots, s_k \}$ is a generating set for $\Gamma(E)$ if and only if for every $x \in X$, there exists $j \in \{ 1, \ldots, k \}$ such that $s_j(x) \neq 0$. In particular, if $s \in \Gamma(E)$, then $\{ s \}$ is a (single) generator for $\Gamma(E)$ if and only if $s$ is nonvanishing on $X$. Consequently, the existence of a single generator for $\Gamma(E)$ is equivalent to the triviality of $E$.
\end{corollary}

Note that the notion of a module frame in $\Gamma(E)$ is different from the common notions of local and global frames of sections of vector bundles, see e.g.\ \cite[p. 257]{lee}. In particular, a global frame for a vector bundle is a basis at every fiber, not just a spanning set.

\section{Time-frequency analysis and Heisenberg modules}\label{sec:tfa}

\subsection{Time-frequency shifts and twisted group \texorpdfstring{$C^*$}{C*}-algebras}

For the rest of the paper, unless otherwise stated, $G$ will denote a second countable, locally compact abelian group. Denote by $\widehat{G}$ the Pontryagin dual of $G$. Given a closed subgroup $H$ of $G$, the \emph{annihilator} of $H$, denoted by $H^\perp$, is the set
\begin{equation}
    H^{\perp} = \{ \omega \in \widehat{G} : \text{ $\omega(x) = 1$ for all $x \in H$ } \} . \label{eq:annihilator} 
\end{equation}
This will always be a closed subgroup of $\widehat{G}$. We have natural isomorphisms $\widehat{H} \cong \widehat{G}/H^\perp$ and $\widehat{G/H} \cong H^\perp$.

We call a subgroup $H$ of $G$ \emph{cocompact} if the quotient $G/H$ is compact. A subgroup $\Lambda$ of $G$ that is both discrete and cocompact is called a \emph{lattice} in $G$. If $\Lambda$ is a lattice in $G$, then $\Lambda^{\perp}$ is a lattice in $\widehat{G}$ \cite[Lemma 3.1]{rieffel88}.\\



Given $x \in G$ and $\omega \in \widehat{G}$, the unitary linear operators $T_x$ and $M_\omega$ on $L^2(G)$ given by
\begin{align}\label{eq:time_frequency}
T_x \xi(t) &= \xi(x^{-1}t) & M_\omega \xi(t) &= \omega(t) \xi(t)
\end{align}
for $\xi \in L^2(G)$, $t \in G$, are called \emph{time shift by $x$} and \emph{frequency shift by $\omega$}, respectively. These two operators satisfy the following commutation relation:
\begin{equation}\label{eq:commutation_relation}
M_\omega T_x = \omega(x) T_x M_\omega.
\end{equation}
A \emph{time-frequency shift} is a combined operator of the form $\pi(x,\omega) = M_\omega T_x$. From \eqref{eq:commutation_relation}, it follows that
\begin{equation}
\pi(x,\omega)\pi(y,\tau) = \overline{\tau(x)} \pi(xy,\omega\tau)
\end{equation}
for $(x,\omega),(y,\tau) \in G \times \widehat{G}$. The function $c \colon (G \times \widehat{G}) \times (G \times \widehat{G}) \to \T$ given by
\begin{equation}
c((x,\omega),(y,\tau)) = \overline{\tau(x)} \label{eq:cocycle}
\end{equation}
satisfies the identities
\begin{align*}
c(z_1,z_2)c(z_1z_2,z_3) &= c(z_1,z_2z_3)c(z_2,z_3) \\
c(1,1) &= 1
\end{align*}
for all $z_1,z_2,z_3 \in G \times \widehat{G}$ and with $1$ denoting the identity element of $G \times \widehat{G}$. Thus, when restricting $c$ to any lattice $\Delta$ in $G \times \widehat{G}$, we obtain a 2-cocycle on $\Delta$. Restricting also $\pi$ to $\Delta$, we obtain a \emph{$c$-projective unitary representation} of the group $\Delta$ on the Hilbert space $L^2(G)$ \cite[Chapter D3]{Wi07}.

Whenever we have a 2-cocycle on a discrete group, we can associate to it the corresponding \emph{twisted group $C^*$-algebra}. In our case, we have the twisted group $C^*$-algebra $C^*(\Delta, c)$, which is the $C^*$-enveloping algebra of the Banach $*$-algebra $\ell^1(\Delta,c)$. As a Banach space, $\ell^1(\Delta,c)$ is $\ell^1(\Delta)$, but the multiplication is \emph{$c$-twisted convolution} and the involution is \emph{$c$-twisted involution}. These operations are defined as follows, where $a,b \in \ell^1(\Delta)$ and $z \in \Delta$:
\begin{align*}
(a * b)(z) &= \sum_{w \in \Delta} c(w,w^{-1}z) a(w)b(w^{-1}z) \\
a^*(z) &= \overline{c(z,z^{-1}) a(z^{-1})} .
\end{align*}
This $C^*$-algebra captures the representation theory of $\Delta$ in the following way: The $c$-projective unitary representations of $\Delta$ are in 1--1 correspondence with nondegenerate $*$-representations of $C^*(\Delta,c)$. In particular, our $c$-projective representation $\pi|_\Delta$ lifts to a representation $C^*(\Delta,c) \to \mathcal{B}(L^2(G))$, which we by abuse of notation also denote by $\pi$. On the dense $*$-subalgebra $\ell^1(\Delta,c)$, $\pi$ is given by
\begin{equation}
\pi(a) = \sum_{z \in \Delta} a(z) \pi(z)
\end{equation}
for $a \in \ell^1(\Delta,c)$. This representation is faithful \cite[Proposition 2.2]{rieffel88}.\\

\subsection{Gabor frames}
As before, let $G$ be a second countable, locally compact abelian group. Let $\Delta$ be a lattice in $G \times \widehat{G}$, and let $\eta \in L^2(G)$. The \emph{Gabor system generated by $\eta$ over $\Delta$} is the set
\begin{equation}
    \mathcal{G}(\eta,\Delta) = \{ \pi(z) \eta : z \in \Delta \} . \label{eq:gabor_system}
\end{equation}
If this set is a frame for the Hilbert space $L^2(G)$ as in \eqref{eq:frame_def}, we call it a \emph{Gabor frame} over $\Delta$ with generator $\eta$. That is, $\mathcal{G}(\eta,\Delta)$ is a Gabor frame if there exist $\lfb, \ufb > 0$ such that
\[ \lfb \| \xi \|_2^2 \leq \sum_{z \in \Delta} | \langle \xi, \pi(z) \eta \rangle |^2 \leq \ufb \| \xi \|_2^2 \]
for all $\xi \in L^2(G)$.

Generalizing, if $\eta_1, \ldots, \eta_k \in L^2(G)$, then the \emph{multiwindow Gabor system} generated by the $\eta_1, \ldots, \eta_k$ over $\Delta$ is the set
\begin{equation}
    \mathcal{G}(\eta_1, \ldots, \eta_k ; \Delta) \coloneqq \{ \pi(z) \eta_j  : z \in \Delta, 1 \leq j \leq k \} . \label{req:multiwindow_system}
\end{equation}
If this set is a frame for $L^2(G)$, then we call it a \emph{multiwindow Gabor frame} over $\Delta$ with generators $\eta_1, \ldots, \eta_k$.

\begin{remark}
It is possible to consider Gabor systems over closed subgroups $\Delta$ of $G \times \widehat{G}$ that are not discrete. This leads to the notion of continuous Gabor systems and frames, see \cite{jakobsen_density}. In our case, we will eventually consider $\Delta$ to be of the form $\Lambda \times \Lambda^{\perp}$, with $\Lambda$ a closed subgroup of $G$. It turns out that for frames to exist over such a closed subgroup of $G \times \widehat{G}$, $\Lambda$ must be a lattice in $G$ \cite[Corollary 5.8]{jakobsen_density}, so we will not lose anything interesting by assuming that $\Delta$ is discrete.
\end{remark}

The \emph{Feichtinger algebra} is the set $S_0(G)$ of all functions $\xi \in L^2(G)$ for which
\begin{equation}
    \int_{G \times \widehat{G}} | \langle \xi, \pi(z) \xi \rangle | \, \dee z < \infty . \label{eq:def_feichtinger}
\end{equation}
The Feichtinger algebra has many different descriptions, see \cite{notnew}. Functions in $S_0(G)$ are continuous.

\subsection{Heisenberg modules}\label{sec:heismodules}

As usual, let $G$ denote a second countable, locally compact abelian group. Let $\Delta$ be a lattice in the time-frequency plane $G \times \widehat{G}$.

We now proceed to review the construction by Rieffel in \cite{rieffel88} of the modules that have been termed \emph{Heisenberg modules}. We follow the approach by Jakobsen and Luef in \cite{gaborheis} and use the Feichtinger algebra $S_0(G)$ instead of the Schwartz--Bruhat space $\mathscr{S}(G)$ from \cite{rieffel88}. Although the Heisenberg module has a natural structure as an imprimitivity bimodule, only the left module structure will be important to us.

By the proof of \cite[Theorem 3.4]{gaborheis}, we have the following:

\begin{proposition}\label{prop:heis}
Let $G$ be a second countable, locally compact abelian group, and let $\Delta$ be a lattice in $G \times \widehat{G}$. Then the Feichtinger algebra $S_0(G)$ can be completed into a full finitely generated left Hilbert $C^*(\Delta,c)$-module which we denote by $\heis{G}{\Delta}$. If $a \in \ell^1(\Delta,c)$ and $\xi \in S_0(G)$, then $a \cdot \xi \in S_0(G)$ and is given by
\begin{equation}
a \cdot \xi = \sum_{z \in \Delta} a(z) \pi(z) \xi \label{eq:module_act}.
\end{equation}
If $\xi, \eta \in S_0(G)$, then $\lhs{\xi}{\eta} \in \ell^1(\Delta,c)$ and is given by
\begin{equation}
\lhs{\xi}{\eta}(z) = \langle \xi, \pi(z) \eta \rangle  \label{eq:inner_prod}
\end{equation}
for $z \in \Delta$.
\end{proposition}

The module $\heis{G}{\Delta}$ is called a \emph{Heisenberg module}. Note that $\heis{G}{\Delta}$ is finitely generated in our case since $\Delta$ is assumed to be a lattice in $G \times \widehat{G}$ (see the proof of \cite[Proposition 3.3]{rieffel88} and \cite[Theorem 3.9]{gaborheis}).

The following is proved in \cite[Proposition 3.2, Proposition 3.6, Theorem 3.11]{AuEn19}, and gives a characterization of generators of Heisenberg modules as exactly multiwindow Gabor frames as in \eqref{req:multiwindow_system}.

\begin{proposition}\label{prop:ausens}
Let $\Delta$ be a lattice in $G \times \widehat{G}$. Then the following hold:
\begin{enumerate}
    \item\label{it:ausens1} There is a continuous embedding of the Heisenberg module $\heis{G}{\Delta}$ into $L^2(G)$ that makes the following diagram commute:
\begin{center}
\begin{tikzpicture}
  \matrix (m) [matrix of math nodes,row sep=3em,column sep=4em,minimum width=2em]
  {
     S_0(G) & \heis{G}{\Delta} \\
      & L^2(G) \\};
  \path[-stealth]
    (m-1-1) edge node {} (m-1-2)
    (m-1-1) edge node {} (m-2-2)
    (m-1-2) edge node {} (m-2-2);
\end{tikzpicture}
\end{center}
    Under this embedding, we have that
    \begin{equation}
        \| \eta \|_{2} \leq \| \eta \|_{\heis{G}{\Delta}} \label{eq:norm_decreasing}
    \end{equation}
    for every $\eta \in \heis{G}{\Delta}$.
    \item\label{it:ausens2} The norm on $\heis{G}{\Delta}$ is given by
    \[ \| \eta \|_{\heis{G}{\Delta}} = \sup_{\| \xi \|_2 = 1} \left( \sum_{z \in \Delta} | \langle \xi, \pi(z) \eta \rangle |^2 \right)^{1/2} \]
    for $\eta \in \heis{G}{\Delta}.$ Thus, $\heis{G}{\Delta}$ can be described as the completion of $S_0(G)$ in the Banach space consisting of those $\eta \in L^2(G)$ for which $\mathcal{G}(\eta,\Delta)$ is a Bessel sequence for $L^2(G)$.
    \item\label{it:ausens3} A subset $\{ \eta_1, \ldots, \eta_k \}$ of $\heis{G}{\Delta}$ is a generating set for $\heis{G}{\Delta}$ if and only if $\mathcal{G}(\eta_1, \ldots, \eta_k ; \Delta)$ is a multiwindow Gabor frame for $L^2(G)$ (see \eqref{eq:gabor_system}).
\end{enumerate}
\end{proposition}

\begin{remark}
Note that in \eqref{eq:norm_decreasing}, there is no appearance of the constant $s(\Delta)^{1/2}$ as in \cite[Proposition 3.2]{AuEn19}. The reason for this is that in the case of a lattice, both traces in the setting of \cite[Convention 3.1]{AuEn19} are finite. Thus, one can consider the trace $\tr_A$ instead of $\tr_B$ when proving the inequality of the $L^2(G)$-norm and the Heisenberg module norm. Since $\tr_A$ is a state in this setting, we have $\| \tr_A \| = 1$, so we obtain \eqref{eq:norm_decreasing}.
\end{remark}

We are now ready to obtain a useful reformulation of \Cref{problem:bl} in terms of Heisenberg modules:

\begin{proposition}\label{prop:balian_reform}
Let $\Delta = \Lambda \times \Lambda^{\perp}$ for a lattice $\Lambda$ in $G$. Then the following are equivalent:
\begin{enumerate}
\item\label{it:hes1} The Heisenberg module $\heis{G}{\Delta}$ is not singly generated.
\item\label{it:hes2} For every $\eta \in \heis{G}{\Delta}$, the Gabor system $\mathcal{G}(\eta,\Delta)$ is not a frame for $L^2(G)$.
\item\label{it:hes3} \Cref{problem:bl} holds for $(G,\Lambda)$. That is, for every $\eta \in S_0(G)$, the Gabor system $\mathcal{G}(\eta, \Delta)$ is not a frame for $L^2(G)$.
\end{enumerate}
\end{proposition}

\begin{proof}
The equivalence of \ref{it:hes1} and \ref{it:hes2} follows immediately from \Cref{prop:ausens} \ref{it:ausens3}, and \ref{it:hes2} implies \ref{it:hes3} since $S_0(G) \subseteq \heis{G}{\Delta}$. It remains to prove that \ref{it:hes3} implies \ref{it:hes1}.

First, note the following: If $\{ \eta_1, \ldots, \eta_k \}$ is a module frame for $\heis{G}{\Delta}$, then using the same procedure as in the proof of \cite[Proposition 3.7]{rieffel88}, one obtains a new module frame $\{ \eta_1', \ldots, \eta_k' \}$ for $\heis{G}{\Delta}$ where $\eta_j' \in S_0(G)$ for $1 \leq j \leq k$. The argument works out because $\ell^1(\Delta,c)$, like the Schwartz--Bruhat space, is spectrally invariant in $C^*(\Delta,c)$. As pointed out in \cite[Theorem 2.2]{projmod}, this is a consequence of \cite[Theorem 3.1]{GrLe04} in the case of $G = \R^n$. However, as pointed out in \cite[p.\ 16]{GrLe04}, the proofs generalize in a straightforward manner to the locally compact abelian case.

Thus, the existence of a one-element module frame (i.e. a single generator) $\eta \in \heis{G}{\Delta}$ implies the existence of a one-element module frame (i.e. single generator) $\eta' \in S_0(G)$. Thus, if no single generators can be found in $S_0(G)$, no single generators can be found in $\heis{G}{\Delta}$, which finishes the proof.
\end{proof}

\section{Vector bundles and quasiperiodic functions}\label{sec:bundles_qp}

\subsection{The Zak transform}\label{sec:zakt}

Throughout this section, we assume that $G$ is a second countable, locally compact abelian group, and that $\Lambda$ is a lattice in $G$.

The \emph{Zak transform} \cite{gelfand,weil} of a function $\xi \in L^2(G)$ with respect to $\Lambda$ is the function $Z \xi = Z_{G,\Lambda} \xi \colon G \times \widehat{G} \to \C$ given by
\begin{equation}
Z \xi(x,\omega) = \sum_{\lambda \in \Lambda} \xi(x\lambda) \omega(\lambda) \label{eq:zak_def}
\end{equation}
for $(x,\omega) \in G \times \widehat{G}$. By \cite[Lemma 3]{kankut}, the Zak transform of a function $\xi \in L^2(G)$ is defined almost everywhere on $G \times \widehat{G}$.

The following result is proved for $G = \R^n$ in \cite[Lemma 8.2.1 (c)]{groechenig}. While the author believes the result for general locally compact abelian groups to be well-known, he was not able to find a reference for it, so it is proved in the appendix, see \Cref{prop:feichtinger_zak_cont_app}.

\begin{proposition}\label{prop:feichtinger_zak_cont}
Let $\Lambda$ be a lattice in a second countable, locally compact abelian group $G$. If $\xi \in S_0(G)$, then $Z_{G,\Lambda}\xi$ is continuous.
\end{proposition}

Now note that for $\xi \in L^2(G)$, $(x,\omega) \in G \times \widehat{G}$ and $(\lambda, \tau) \in \Lambda \times \Lambda^{\perp}$, we have the following:
\begin{align*}
Z\xi(x\lambda,\omega\tau) &= \sum_{\lambda' \in \Lambda} \xi(x\lambda \lambda') (\omega\tau)(\lambda') \\
&= \sum_{\lambda'\in \Lambda} \xi(x \lambda') \omega(\lambda' \lambda^{-1}) \\
&= \overline{\omega(\lambda)} Z\xi(x,\omega).
\end{align*}
Thus, if $\xi \in S_0(G)$, then by \Cref{prop:feichtinger_zak_cont} and the above calculation, $Z\xi$ is an element of the following function space:

\begin{definition}\label{def:quasi}
Denote by $\sec{G}{\Lambda}$ the complex vector space of continuous functions $F \colon G \times \widehat{G} \to \C$ satisfying the relation
\begin{equation}
F(x \lambda, \omega \tau) = \overline{\omega(\lambda)} F(x,\omega) \label{eq:quasi}
\end{equation}
for all $(x,\omega) \in G \times \widehat{G}$ and $(\lambda, \tau) \in \Lambda \times \Lambda^{\perp}$.
\end{definition}

The relation in \eqref{eq:quasi} is referred to as a \emph{quasiperiodicity relation} in the literature \cite{kankut}, so we will call elements of $\sec{G}{\Lambda}$ continuous, quasiperiodic functions on $G \times \widehat{G}$.

Define $X_{G,\Lambda}$ to be the space
\begin{equation}
    X_{G,\Lambda} = (G/\Lambda) \times (\widehat{G}/\Lambda^{\perp}) . \label{eq:spacex}
\end{equation}
Since $\Lambda$ is a lattice in $G$, this is a compact space. The linear space $\sec{G}{\Lambda}$ becomes a left $C(X_{G,\Lambda})$-module with the pointwise action
\begin{equation}
(f \cdot F)(x,\omega) = f([x],[\omega]) F(x,\omega)
\end{equation}
for $(x,\omega) \in G \times \widehat{G}$, $f \in C(X_{G,\Lambda})$ and $F \in \sec{G}{\Lambda}$. Moreover, we define a $C(X_{G,\Lambda})$-valued inner product on $\sec{G}{\Lambda}$ by
\begin{equation}
\lhs{F}{G}([x],[\omega]) = F(x,\omega) \overline{G(x,\omega)}
\end{equation}
for $(x,\omega) \in G \times \widehat{G}$ and $F,G \in \sec{G}{\Lambda}$. The above structure is easily seen to satisfy axioms (i)-(iii) in the definition of a Hilbert $C(X_{G,\Lambda})$-module from \Cref{sec:hilbertmodules}. The norm induced by the $C(X_{G,\Lambda})$-valued inner product is just the supremum norm, so any Cauchy sequence $(F_n)_{n \in \N}$ of continuous, quasiperiodic functions will have a uniform limit $F \in C_b(G \times \widehat{G})$. Continuity then gives that the limit is also quasiperiodic, showing that $\sec{G}{\Lambda}$ is in fact a Hilbert $C^*$-module over $C(X_{G,\Lambda})$.

Our aim is to show that $\sec{G}{\Lambda}$ is finitely generated as a $C(X_{G,\Lambda})$-module, so that it corresponds to a Hermitian vector bundle over $X_{G,\Lambda}$ by \Cref{prop:serre_swan}. To do this, we begin by characterizing finite module frames of $\sec{G}{\Lambda}$:

\begin{proposition}\label{prop:frame_in_qp}
Let  $\{ G_1, \ldots, G_k \}$ be a subset of $\sec{G}{\Lambda}$. Then the following hold:
\begin{enumerate}
\item\label{it:frameqp1} $\{ G_1, \ldots, G_k \}$ is a module frame in $\sec{G}{\Lambda}$ if and only if for every $(x,\omega) \in G \times \widehat{G}$, there exists $1 \leq j \leq k$ such that $G_j(x,\omega) \neq 0$.
\item\label{it:frameqp3} $\{ G_1, \ldots, G_k \}$ is a normalized tight module frame in $\sec{G}{\Lambda}$ if and only if
\[ \sum_{j=1}^k |G_j(x,\omega)|^2 = 1 \]
for every $(x,\omega) \in G \times \widehat{G}$.
\end{enumerate}
In particular, a single generator $\{ G\}$ for $\sec{G}{\Lambda}$ is just a nonvanishing continuous quasiperiodic function on $G \times \widehat{G}$.
\end{proposition}

\begin{proof}
Note that the module frame condition from \eqref{eq:frame_def} on $\{ G_1, \ldots, G_k \}$ translates into the existence of $\lfb,\ufb > 0$ such that
\[ \lfb | F(x,\omega)|^2 \leq \sum_{j=1}^k F(x,\omega) \overline{ G_j(x,\omega)} \overline{F(x,\omega)} G_j(x,\omega) \leq \ufb | F(x,\omega)|^2 \]
for all $F \in \sec{G}{\Lambda}$ and all $(x,\omega) \in G \times \widehat{G}$. Hence, cancelling $|F(x,\omega)|^2$, we obtain the condition
\[ \lfb \leq \sum_{j=1}^k |G_j(x,\omega)|^2 \leq \ufb \]
for all $(x,\omega) \in G \times \widehat{G}$. Since each $|G_j|^2$ can be seen as a continuous function on the compact space $X_{G,\Lambda}$, it is clear that the upper frame condition automatically holds by the extreme value theorem. For the same reason, the lower frame condition is equivalent to the expression $\sum_{j=1}^k |G_j(x,\omega)|^2$ being nonvanishing, which is equivalent to (\ref{it:frameqp1}). Now (\ref{it:frameqp3}) follows directly by considering $\lfb = \ufb=1$ in this situation.
\end{proof}

The following proposition shows that $\sec{G}{\Lambda}$ is finitely generated. In fact, one can always find a module frame consisting of images of elements in $S_0(G)$ under the Zak transform. Recall that $T_x$ denotes a time shift as in \eqref{eq:time_frequency}.

\begin{proposition}\label{prop:qp_fingen}
Let $U$ be an open neighbourhood of $1$ in $G$ such that the collection $\{ \lambda U : \lambda \in \Lambda \}$ is pairwise disjoint, let $K \subseteq U$ be a nonempy compact set, and let $\xi$ a function in $S_0(G)$ satisfying $\xi|_K = 1$ and $\supp(\xi) \subseteq U$ (the existence of such a function is guaranteed by \Cref{lem:cutoff}). Then there exist $x_1, \ldots, x_k \in G$ such that
\[ \{ Z_{G,\Lambda}(T_{x_1}\xi), \ldots, Z_{G,\Lambda}(T_{x_k}\xi) \} \]
is a module frame for $\sec{G}{\Lambda}$. Consequently, $\sec{G}{\Lambda}$ is a finitely generated left Hilbert $C(X_{G,\Lambda})$-module.
\end{proposition}

\begin{proof}
Let $(x,\omega) \in G \times \widehat{G}$. Then $xU$ is an open neighbourhood of $x$, and the sets $\{ \lambda x U : \lambda \in \Lambda \}$ are pairwise disjoint. Moreover, if $y \in xK$ and $\lambda \in \Lambda$ satisfy $\xi(x^{-1}y\lambda) \neq 0$, then $x^{-1}y\lambda \in \supp(\xi) \subseteq U$. Thus $x^{-1}y \in U \cap (\lambda^{-1}U)$ which forces $\lambda = 1$. Consequently, we have that
\begin{equation}
Z(T_x\xi)(y,\omega) = \sum_{\lambda \in \Lambda} \xi(x^{-1}y\lambda) \omega(\lambda) = \xi(x^{-1}y) = 1. \label{eq:rel_zak}
\end{equation}
By \eqref{eq:quasi} and \eqref{eq:rel_zak}, we have that that
\[ Z(T_x \xi)(y\lambda', \omega) = \overline{\omega(\lambda')} Z(T_x \xi)(y,\omega) = \overline{\omega(\lambda')}  \]
whenever $y \in xK$ and $\lambda' \in \Lambda$. This shows that the function $Z(T_x \xi)$ is nonvanishing on $xK^{\circ}\Lambda \subseteq x K \Lambda$. Now $\{ xK^{\circ} : x \in G \}$ is an open cover of $G$, and $\Lambda$ is cocompact in $G$, so we can find a finite cover of $G$ of the form $\{ x_j K^{\circ} \Lambda : 1 \leq j \leq k \}$ for $x_1, \ldots, x_k \in G$. For each $1 \leq j \leq k$, $F_j \coloneqq Z_{G,\Lambda}(T_{x_j} \xi)$ is nonvanishing on $x_j K^{\circ} \Lambda$. By \Cref{prop:feichtinger_zak_cont}, each $F_j$ is in $\sec{G}{\Lambda}$. Thus, by \Cref{prop:frame_in_qp}, the set $\{ F_1, \ldots, F_k \}$ is a module frame for $\sec{G}{\Lambda}$.

Since a module frame is the same as a generating set by \Cref{prop:frame_char}, $\sec{G}{\Lambda}$ is finitely generated as a left $C(X_{G,\Lambda})$-module.
\end{proof}

\subsection{The line bundle \texorpdfstring{$E_{G,\Lambda}$}{EG,lambda}}\label{sec:bundle}

In this section, we construct the line bundle mentioned in \Cref{thm:intro2} and \Cref{thm:intro1} in the introduction.

We begin with the left Hilbert $C(X_{G,\Lambda})$-module $\sec{G}{\Lambda}$ defined in \Cref{def:quasi}, which was shown to be finitely generated in \Cref{prop:qp_fingen}. By \Cref{prop:serre_swan}, there exists a Hermitian vector bundle $E \to X_{G,\Lambda}$ such that $\Gamma(E) \cong \sec{G}{\Lambda}$. In fact, letting $\{ G_1, \ldots, G_k \}$ be a normalized tight frame for $\sec{G}{\Lambda}$, we have that $\sec{G}{\Lambda} \cong C(X_{G,\Lambda})^k P$, where
\[ P = \begin{pmatrix} \lhs{G_1}{G_1} & \cdots & \lhs{G_1}{G_k} \\ \vdots & \ddots & \cdots \\ \lhs{G_k}{G_1} & \cdots & \lhs{G_k}{G_k} \end{pmatrix} = \begin{pmatrix} G_1 \overline{G_2} & \cdots & G_1 \overline{G_k} \\ \vdots & \ddots & \cdots \\ G_k \overline{G_1} & \cdots & G_k \overline{G_k} \end{pmatrix} . \]
It then follows from \Cref{prop:serre_swan} that
\[ E \cong \{ ([x],[\omega],v) \in X_{G,\Lambda} \times \C^k : vP([x],[\omega]) = v \} .\]
The condition $vP([x],[\omega]) = v$ for $([x],[\omega]) \in X_{G,\Lambda}$ and $v = (v_j)_{j=1}^k \in \C^k$ translates into
\[ v_j = \sum_{i=1}^k G_i(x,\omega) \overline{G_j(x,\omega)} v_i = \overline{G_j(x,\omega)} \sum_{i=1}^k G_i(x,\omega) v_i \]
for each $1 \leq j \leq k$. Thus, $v_j / \overline{G_j(x,\omega)}$ is constant in $j$, so we can write $v$ as a complex multiple of the vector $(\overline{G_j(x,\omega)})_{j=1}^k$. It follows that we can describe $E$ as
\[ E \cong \{ ([x],[\omega], (z \overline{G_j(x,\omega)})_{j=1}^k) : x \in G, \omega \in \widehat{G}, z \in \C \} .\]
This shows that $E$ is in fact a line bundle. Now we want to obtain a description of $E$ independent of the chosen module frame. Note that the map $G \times \widehat{G} \times \C \to E_{G,\Lambda}$ given by
\[ (x,\omega,z) \mapsto ([x],[\omega],(z \overline{G_j(x,\omega)})_{j=1}^k) \]
is continuous and surjective, and that $(x,\omega,z)$ and $(x',\omega',z')$ has the same image under this map if and only if there exist $\lambda \in \Lambda$ and $\tau \in \Lambda^{\perp}$ such that
\begin{align*}
x' &= x\lambda, \\
\omega' &= \omega \tau, \\
z' &= \overline{\omega(\lambda)} z.
\end{align*}
Thus, we get a continuous bijection
\[ E_{G,\Lambda} \coloneqq (G \times \widehat{G} \times \C)/\sim \;\; \to E \]
where $\sim$ is the equivalence relation defined above. If we transfer the vector bundle structure from $E$ to $E_{G,\Lambda}$ via this map, then the projection $\pi \colon E_{G,\Lambda} \to X_{G,\Lambda}$ is given by $\pi([x,\omega,z]) = ([x],[\omega])$ for $[x,\omega,z] \in \tilde{E}_{G,\Lambda}$. The vector space structure on the fiber
\[ \pi^{-1}([x],[\omega]) = \{ [x,\omega,z] : z \in \C \} \]
is given by
\begin{align}
[x,\omega,z] + [x,\omega,w] &= [x,\omega,z+w], \label{eq:vb1} \\
\mu [x,\omega,z] &= [x,\omega,\mu z] \label{eq:vb2}
\end{align}
for $x \in G$, $\omega \in \widehat{G}$ and $z,w,\mu \in \C$. Furthermore, the inner product on the same fiber is given by
\begin{equation}
\langle [x,\omega,z], [x,\omega,w] \rangle_{([x],[\omega])} = z \overline{w} . \label{eq:vb3}
\end{equation}
Since the map $E_{G,\Lambda} \to E$ restricts to a linear isomorphism on each fiber, it is a vector bundle isomorphism by \cite[Lemma 2.3]{milnor}. We have thus obtained a description of $E$ independent on a chosen module frame. We summarize our results in the following proposition:

\begin{proposition}\label{prop:hilbert_sec}
Let $E_{G,\Lambda}$ denote the quotient of $G \times \widehat{G} \times \C$ by the equivalence relation given by $(x,\omega,z) \sim (x',\omega',z')$ if and only if there exist $\lambda \in \Lambda$ and $\tau \in \Lambda^{\perp}$ such that
\begin{align*}
x' &= x\lambda, \\
\omega' &= \omega \tau, \\
z' &= \overline{\omega(\lambda)} z.
\end{align*}
Then $E_{G,\Lambda}$ has the structure of a Hermitian line bundle with operations given in \eqref{eq:vb1}, \eqref{eq:vb2} and \eqref{eq:vb3}, and the map $\Psi \colon \sec{G}{\Lambda} \to \Gamma(E_{G,\Lambda})$ given by
\[ \Psi(F)([x],[\omega]) = [x,\omega,F(x,\omega)] \]
is an isomorphism of left Hilbert $C(X_{G,\Lambda})$-modules.
\end{proposition}

\begin{example}\label{ex:balanbundle}
We compare the bundle $E_{G,\Lambda}$ in the case of $G = \R$ and $\Lambda = \Z$ to the bundle $\xi$ in \cite{balan} in the case of $p=1$. Then $1 = \alpha \beta = p/q = 1$ so that $q=1$, and thus we can set $r_0 = 1$ and $n_0 = 0$ as well (see equation (21)). The matrix $E(t)$ used to define $\xi$ in Equations (32) and (33) in \cite{balan} becomes the $1 \times 1$ matrix $E(t) = e^{-2\pi i t}$, so the Equations (32) and (33) in \cite{balan} reduce to the relations that define $E_{\R,\Z}$, modulo a sign. This shows that $\xi$ is the dual bundle of $E_{\R,\Z}$. However, since the Picard group of $\T^2$ (the abelian group of line bundles under tensor product where inversion is given by taking dual bundles) is isomorphic to $\Z$, a line bundle over $\T^2$ is nontrivial if and only if its dual bundle is nontrivial. Hence the nontriviality of $\xi$ is equivalent to the nontriviality of $E_{\R,\Z}$, and the nontriviality of the former bundle was demonstrated in \cite{balan}.
\end{example}

\section{Connecting Heisenberg modules and vector bundles}\label{sec:connecting}

\subsection{The Zak transform as an isomorphism of Hilbert C*-modules}\label{sec:zak}

We now combine the Heisenberg modules from \Cref{sec:heismodules} and the setting of \Cref{problem:bl}. We consider the Heisenberg module $\heis{G}{\Delta}$ over the twisted group $C^*$-algebra $C^*(\Delta,c)$, where the lattice $\Delta \subseteq G \times \widehat{G}$ is of the form $\Delta =  \Lambda \times \Lambda^{\perp}$ for a lattice $\Lambda$ in $G$. The 2-cocycle $c$ from \eqref{eq:cocycle} restricted to $\Delta$ is given by
\[ c((\lambda,\tau),(\lambda',\tau')) = \overline{\tau'(\lambda)} = 1 \]
since $\lambda \in \Lambda$ and $\tau \in \Lambda^{\perp}$. Thus, the 2-cocycle is constantly equal to 1 on $\Lambda \times \Lambda^{\perp}$ and the twisted group $C^*$-algebra $C^*(\Lambda \times \Lambda^{\perp},c)$ becomes the un-twisted group $C^*$-algebra $C^*(\Lambda \times \Lambda^{\perp})$.

The group $\Lambda \times \Lambda^{\perp}$ is abelian, and its Pontryagin dual can be identified with $X_{G,\Lambda}$ from \Cref{sec:bundle} as follows:
\begin{align*}
\widehat{\Lambda \times \Lambda^{\perp}} &\cong (G \times \widehat{G})/((\Lambda \times \Lambda^{\perp})^{\perp} && \text{using the identity $\widehat{H} \cong \widehat{G}/H^{\perp}$ } \\
&\cong (\widehat{G}/\Lambda^{\perp}) \times \left( G / \Lambda \right) \\
&\cong (G/\Lambda) \times (\widehat{G}/\Lambda^{\perp}) \\
&= X_{G,\Lambda}. 
\end{align*}
Note that we flipped the two spaces in the last homeomorphism. Thus, we get $C^*(\Lambda \times \Lambda^{\perp}) \cong C(X_{G,\Lambda})$ via the Fourier transform, see \cite[Proposition 3.1]{Wi07}. We will use the Fourier transform $\phi \colon C^*(\Lambda \times \Lambda^{\perp}) \to C(X_{G,\Lambda})$ given by
\begin{equation}
    \phi(a)([x],[\omega]) = \sum_{\lambda \in \Lambda, \tau \in \Lambda^{\perp}} a(\lambda,\tau) \omega(\lambda) \tau(x) \label{eq:fourier_phi}
\end{equation}
for $a \in \ell^1(\Lambda \times \Lambda^{\perp})$, $x \in G$, $\omega \in \widehat{G}$, to identify the two $C^*$-algebras.

Now the Heisenberg module $\heis{G}{\Lambda \times \Lambda^{\perp}}$ is a finitely generated Hilbert $C^*$-module over $C^*(\Lambda \times \Lambda^{\perp}) \cong C(X_{G,\Lambda})$ by \Cref{prop:heis}. By the Serre--Swan theorem (\Cref{prop:serre_swan}), we know that $\heis{G}{\Lambda \times \Lambda^{\perp}}$ must be isomorphic to the continuous sections of a Hermitian vector bundle over $X_{G,\Lambda}$. We now prove that the bundle in question is $E_{G,\Lambda}$ from \Cref{sec:bundle}, and that the precise isomorphism is implemented by the Zak transform from \Cref{sec:zakt}:

\begin{theorem}\label{thm:module_iso}
Let $\Lambda$ be a lattice in a second countable, locally compact abelian group $G$. Consider the Heisenberg module $\heis{G}{\Lambda \times \Lambda^{\perp}}$ over $C^*(\Lambda \times \Lambda^{\perp})$, with $S_0(G)$ as a dense $\ell^1(\Lambda \times \Lambda^{\perp})$-submodule. Let $\phi$ denote the Fourier transform from \eqref{eq:fourier_phi}. For $a \in \ell^1(\Lambda \times \Lambda^{\perp})$ and $\xi,\eta \in S_0(G)$, we have that
\begin{align}
Z(a \cdot \xi) &= \phi(a) \cdot Z(\xi) \label{eq:zak_mod1}, \\
\phi( \lhs{\xi}{\eta}) &= \lhs{Z\xi}{Z\eta} \label{eq:zak_mod2} .
\end{align}
Using $\phi$ to identify $C^*(\Lambda \times \Lambda^{\perp})$ with $C(X_{G,\Lambda})$, the Zak transform given in \eqref{eq:zak_def} is an isomorphism of left Hilbert modules
\[  Z_{G,\Lambda}: \heis{G}{\Lambda \times \Lambda^{\perp}} \to \sec{G}{\Lambda}. \]
\end{theorem}

\begin{proof}
First, \Cref{prop:feichtinger_zak_cont} ensures that $Z$ maps $S_0(G)$ into $\sec{G}{\Lambda}$. To show \eqref{eq:zak_mod1}, we begin by letting $a = \delta_{(\lambda_0,\tau_0)} \in \ell^1(\Lambda \times \Lambda^{\perp})$ for $(\lambda_0,\tau_0) \in \Lambda \times \Lambda^{\perp}$ and $\xi \in S_0(G)$. Then
$\phi$ maps $\delta_{(\lambda_0,\tau_0)}$ to the function on $(G/\Lambda) \times (\widehat{G}/\Lambda^{\perp})$ given by $([x],[\omega]) \mapsto \omega(\lambda_0) \tau_0(x)$. Therefore, letting $x \in G$ and $\omega \in \widehat{G}$, we obtain
\begin{align*}
Z( \delta_{(\lambda_0,\tau_0)} \cdot \xi)(x,\omega) &= \sum_{\lambda \in \Lambda} \pi(\lambda_0,\tau_0) \xi(x\lambda) \omega(\lambda) && \text{by \eqref{eq:module_act}} \\
&= \sum_{\lambda \in \Lambda} \xi(x \lambda \lambda_0^{-1}) \tau_0(x\lambda) \omega(\lambda)  \\
&= \sum_{\lambda \in \Lambda} \xi(x \lambda) \tau_0(x \lambda \lambda_0) \omega(\lambda\lambda_0) && \text{via $\lambda \mapsto \lambda \lambda_0$ } \\
&= \omega(\lambda_0) \tau_0(x) \sum_{\lambda \in \Lambda} \xi(x \lambda) \omega(\lambda) \\
&= (\phi(\delta_{(\lambda_0,\tau_0)}) \cdot (Z\xi)) (x,\omega).
\end{align*}
By linearity and continuity, this proves \eqref{eq:zak_mod1} for all $a \in \ell^1(\Lambda \times \Lambda^{\perp})$.

We move on to proving \eqref{eq:zak_mod2}. Let $\xi,\eta \in S_0(G)$. Denote by $\mathcal{F}$ the usual Fourier transform $S_0(G) \to S_0(\widehat{G})$. Since the Poisson summation formula holds for functions in $S_0(G)$ \cite[Theorem 5.7 (iii)]{notnew} and $S_0(G)$ is an algebra under pointwise multiplication, we have that the Poisson summation formula holds for the function $t \mapsto \xi(tx) \overline{\eta(\lambda^{-1}tx)}$, with $x \in G$, $\omega \in \widehat{G}$. We do the following calculation, where the Poisson summation formula is applied in the fifth equality:
{\allowdisplaybreaks
\begin{align*}
\phi(\lhs{\xi}{\eta})([x],[\omega]) &= \sum_{\lambda \in \Lambda, \tau \in \Lambda^{\perp}} \langle \xi, \pi(\lambda,\tau)\eta) \rangle \omega(\lambda) \tau(x) \\
&= \sum_{\lambda \in \Lambda, \tau \in \Lambda^{\perp}} \int_G \xi(t) \overline{ \eta(\lambda^{-1} t) \tau(t)} \omega(\lambda) \tau(x) \, \dee t \\
&= \sum_{\lambda \in \Lambda, \tau \in \Lambda^{\perp}} \int_G \xi(tx) \overline{ \eta(\lambda^{-1} tx) \tau(t)} \omega(\lambda) \, \dee t && \text{ via $t \mapsto tx$ } \\
&= \sum_{\lambda \in \Lambda} \sum_{\tau \in \Lambda^{\perp}} \mathcal{F}\big[t \mapsto \xi(tx) \overline{ \eta(\lambda^{-1} tx)}\big](\tau) \omega(\lambda) \\
&= \sum_{\lambda \in \Lambda} \sum_{\lambda' \in \Lambda} \xi(\lambda' x) \overline{\eta(\lambda^{-1} \lambda' x)} \omega(\lambda) && \text{by Poisson summation} \\
&= \sum_{\lambda,\lambda' \in \Lambda} \xi(x \lambda') \overline{ \eta(x \lambda)} \overline{\omega(\lambda)} \omega(\lambda') && \text{via $\lambda \mapsto \lambda \lambda'$ } \\
&= Z\xi(x,\omega) \overline{Z\eta(x,\omega)}.
\end{align*}}
This proves \eqref{eq:zak_mod2}.

We have now established that the Zak transform is linear and inner product preserving (hence an isometry) on $S_0(G)$ and preserves the action of $\ell^1(\Lambda \times \Lambda^{\perp})$ on $S_0(G)$.

Now take $\xi \in \heis{G}{\Lambda \times \Lambda^{\perp}}$, and let $(\xi_n)_n$ be a sequence in $S_0(G)$ such that $\| \xi - \xi_n \|_{\heis{G}{\Lambda \times \Lambda^{\perp}}} \to 0$. Since $Z$ is an isometry on $S_0(G)$, we have that
\[ \| Z \xi_m - Z \xi_n \|_{\infty} = \| Z(\xi_m - \xi_n) \|_{\infty} = \| \xi_m - \xi_n \|_{\heis{G}{\Lambda \times \Lambda^{\perp}}} \]
for $m,n \in \N$. Since $(\xi_n)_n$ converges in the $\heis{G}{\Lambda \times \Lambda^{\perp}}$-norm, we have that $(Z \xi_n)_n$ converges uniformly to a function $F \in \sec{G}{\Lambda}$.

By \cite[Remark 4]{kankut}, the Zak transform can be viewed as a unitary map from $L^2(G)$ to $L^2(B \times B')$, where $B$ and $B'$ are any fundamental domains for $\Lambda$ in $G$ and $\Lambda^{\perp}$ in $\widehat{G}$, respectively. Hence we have, using \eqref{eq:norm_decreasing}, that
\[ \| Z \xi - Z \xi_n \|_{L^2(B \times B')} = \| \xi - \xi_n \|_{L^2(G)} \leq \| \xi - \xi_n \|_{\heis{G}{\Lambda \times \Lambda^{\perp}}} \to 0, \]
so $Z \xi_n \to Z \xi$ in the $L^2$-norm on $B \times B'$. But we also have that
\[ \| Z \xi_n - F \|_{L^2(B \times B')} \leq \| Z \xi_n - F \|_{L^{\infty}(B \times B')} \to 0 \]
which means that $Z \xi = F$ on $B \times B'$. By quasiperiodicity, they must be equal on the whole of $G \times \widehat{G}$. We have thus shown that the Zak transform maps $\heis{G}{\Lambda \times \Lambda^{\perp}}$ into $\sec{G}{\Lambda}$, and by continuity it becomes an inner product preserving $C^*(\Delta, c)$-linear map $\heis{G}{\Lambda \times \Lambda^{\perp}} \to \sec{G}{\Lambda}$.

It remains to show that $Z$ is surjective. Since $Z$ is an isometry, it has closed range, so it suffices to show that $Z(S_0(G))$ is dense in $\sec{G}{\Lambda}$. We show this in the following lemma:
\end{proof}

\begin{lemma}\label{lem:feicht_dense}
The image of $S_0(G)$ under the Zak transform is dense in $\sec{G}{\Lambda}$.
\end{lemma}

\begin{proof}
By \Cref{prop:fingen_module}, we can find a module frame $\{G_1, \ldots, G_k \}$ for $\sec{G}{\Lambda}$ where $G_j = Z \xi_j$ for functions $\xi_j \in S_0(G)$, $1 \leq j \leq k$.

Now suppose $F \in \sec{G}{\Lambda}$. Then since $\{ G_1, \ldots, G_k \}$ is a module frame, it is a generating set for $\sec{G}{\Lambda}$ by \Cref{prop:fingen_module}. Thus, there exist functions $f_1, \ldots, f_k \in C(X_{G,\Lambda})$ such that $F = \sum_j f_j \cdot G_j$. Now since $\ell^1(\Lambda \times \Lambda^{\perp})$ is dense in $C^*(\Lambda \times \Lambda^{\perp})$ and $\phi(C^*(\Lambda \times \Lambda^{\perp})) = C(X_{G,\Lambda})$, we can find sequences $(a_{i,j})_{i=1}^{\infty}$ in $\ell^1(\Lambda \times \Lambda^{\perp})$ for each $1 \leq j \leq k$ such that $\lim_i \phi(a_{i,j}) = f_j$ in the sup-norm. By \eqref{eq:zak_mod1}, we then have that
\[ F = \lim_i \sum_j \phi(a_{i,j}) \cdot G_j = \lim_i Z \left( \sum_j a_{i,j} \cdot \xi_j \right). \]
By \Cref{prop:heis} \eqref{eq:module_act}, we have that $\sum_j a_{i,j} \cdot \xi_j$ is in $S_0(G)$. this shows that $F$ is in the closure of $Z(S_0(G))$, which finishes the proof.
\end{proof}

As a consequence, we obtain the following description of the Heisenberg module $\heis{G}{\Lambda \times \Lambda^{\perp}}$ in terms of the Zak transform.

\begin{proposition}\label{prop:cont_zak_desc}
The Heisenberg module $\heis{G}{\Lambda \times \Lambda^{\perp}}$ consists exactly of the functions $\xi \in L^2(G)$ for which $Z_{G,\Lambda}\xi$ is continuous.
\end{proposition}

\begin{proof}
Let $\xi \in L^2(G)$. Then $Z \xi$ is continuous if and only if $Z \xi \in \sec{G}{\Lambda}$. Since the Zak transform maps $\heis{G}{\Lambda \times \Lambda^{\perp}}$ bijectively onto $\sec{G}{\Lambda}$ by \Cref{thm:module_iso}, it follows that $\xi \in \mathcal{E}$ if and only if $Z \xi$ is continuous.
\end{proof}

We are now in position to prove one of our main results:

\begin{theorem}\label{thm:bundle_bal}
Let $G$ be a second countable, locally compact abelian group, and let $\Lambda$ be a lattice in $G$. Then the following are equivalent:
\begin{enumerate}
\item\label{it:nontriv} The vector bundle $E_{G,\Lambda}$ is nontrivial.
\item\label{it:nonsing} The Heisenberg module $\heis{G}{\Lambda \times \Lambda^{\perp}}$ is not singly generated.
\item\label{it:quasi_zero} Every continuous map $F \colon G \times \widehat{G} \to \C$ satisfying
\begin{align*}
F(x\lambda,\omega\tau) &= \overline{\omega(\lambda)} F(x,\omega)
\end{align*}
for all $(x, \omega) \in G \times \widehat{G}$ and $(\lambda, \tau) \in \Lambda \times \Lambda^{\perp}$, must have a zero.
\item\label{it:zak_cont_zero} Whenever $\xi \in L^2(G)$ is such that $Z_{G,\Lambda} \xi$ is continuous, then $Z_{G,\Lambda} \xi$ has a zero.
\item\label{it:zak_cont_gabor} Whenever $\xi \in L^2(G)$ is such that $Z_{G,\Lambda} \xi$ is continuous, then $\mathcal{G}(\xi,\Lambda \times \Lambda^{\perp})$ is not a Gabor frame for $L^2(G)$.
\item\label{it:feicht_gabor} \Cref{problem:bl} holds for $(G,\Lambda)$. That is, whenever $\xi \in S_0(G)$, then $\mathcal{G}(\xi,\Lambda \times \Lambda^{\perp})$ is not a Gabor frame for $L^2(G)$.
\end{enumerate}
\end{theorem}

\begin{proof}
By \Cref{prop:hilbert_sec} and \Cref{thm:module_iso}, the three modules $\heis{G}{\Lambda \times \Lambda^{\perp}}$, $\sec{G}{\Lambda}$ and $\Gamma(E_{G,\Lambda})$ are all isomorphic, so if one of them is singly generated, then all of them are.

By \Cref{cor:singlegen_bundle}, (\ref{it:nontriv}) is equivalent to the statement that $\Gamma(E_{G,\Lambda})$ is not singly generated, which is again equivalent to (\ref{it:nonsing}).

By \Cref{prop:frame_in_qp}, (\ref{it:quasi_zero}) is equivalent to the statement that $\sec{G}{\Lambda}$ is not singly generated. By \Cref{prop:ausens} (\ref{it:ausens3}) and \Cref{prop:cont_zak_desc}, (\ref{it:zak_cont_gabor}) is equivalent to saying that there are no single generators for $\heis{G}{\Lambda \times \Lambda^{\perp}}$. Since every $F \in \sec{G}{\Lambda}$ is of the form $Z\xi$ for some $\xi \in \heis{G}{\Lambda \times \Lambda^{\perp}}$, (\ref{it:quasi_zero}) is equivalent to (\ref{it:zak_cont_zero}).

Finally, (\ref{it:zak_cont_gabor}) and (\ref{it:feicht_gabor}) are equivalent to (\ref{it:nonsing}) by \Cref{prop:balian_reform}, again using the description of $\heis{G}{\Lambda \times \Lambda^{\perp}}$ in \Cref{prop:cont_zak_desc}.
\end{proof}

It is already known from \cite[Remark 6]{kankut} that (\ref{it:quasi_zero}), (\ref{it:zak_cont_zero}) and (\ref{it:zak_cont_gabor}) in \Cref{thm:bundle_bal} are equivalent. The proofs we present here however, are new.

\subsection{Determining (non)triviality of \texorpdfstring{$E_{G,\Lambda}$}{EG,lambda} for some examples}\label{sec:calc}

Now that we have proved our main result, \Cref{thm:bundle_bal}, we turn to determine the triviality or nontriviality of the bundle $E_{G,\Lambda}$ for specific $(G,\Lambda)$. By \Cref{thm:bundle_bal}, this will then give the validity or nonvalidity of the Balian--Low theorem for $S_0(G)$ over the lattice $\Lambda \times \Lambda^{\perp}$ (\Cref{problem:bl}) in these settings.

We begin with the following result, which describes a kind of functoriality for the construction $(G,\Lambda) \mapsto E_{G,\Lambda}$.

\begin{proposition}\label{prop:iso_lattices}
Let $G$ and $H$ be second countable, locally compact abelian groups. Let $\Lambda$ be a lattice in $G$ and let $\Gamma$ be a lattice in $H$. Suppose $\phi \colon G \to H$ is a topological group isomorphism such that $\phi(\Lambda) = \Gamma$. Then the spaces $X_{G,\Lambda}$ and $X_{H,\Gamma}$ are homeomorphic, and the vector bundles $E_{G,\Lambda}$ and $E_{H,\Gamma}$ are isomorphic.
\end{proposition}

\begin{proof}
Define $h \colon X_{G,\Lambda} \to X_{H,\Gamma}$ by $h([x],[\omega]) = ([\phi(x)],[\omega \circ \phi^{-1}])$. Then it is straightforward to show that $h$ is a continuous bijection, hence a homeomorphism because the spaces are compact Hausdorff. Furthermore, define $\psi \colon E_{G,\Lambda} \to E_{H,\Gamma}$ by $\psi([x,\omega,z]) = [\phi(x), \omega \circ \phi^{-1}, z]$. This map is well-defined, as for $x \in G$, $\omega \in \widehat{G}$, $\lambda \in \Lambda$, $\tau \in \Lambda^{\perp}$ and $z \in \C$, we have that
\begin{align*}
\psi([x\lambda,\omega\tau,\overline{\omega(\lambda)}z] &= [\phi(x\lambda), (\omega\tau) \circ \phi^{-1}, \overline{\omega(\lambda)}z ] \\
&= [ \phi(x)\phi(\lambda), (\omega \circ \phi^{-1})(\tau \circ \phi^{-1}), \overline{\omega(\lambda)} z] \\
&= [ \phi(x), \omega \circ \phi^{-1}, \omega \circ \phi^{-1}(\phi(\lambda)) \overline{\omega(\lambda)}z ] \\
&= [\phi(x), \omega \circ \phi^{-1}, z ] \\
&= \psi([x,\omega,z]) .
\end{align*}
Moreover, it is continuous and restricts to an isomorphism on each fiber, so it is an isomorphism of vector bundles.
\end{proof}

In the following example we look at the vector bundles associated to the groups $G = \R^n$ for $n \in \N$.

\begin{example}\label{ex:realbundle}
Let $G = \R^n$. Any lattice in $\R^n$ is of the form $A\Z^n$ for an invertible real $n \times n$ matrix $A$. Since multiplication by $A$ is a topological isomorphism on $\R^n$ that maps $\Z^n$ to $A \Z^n$, we have by \Cref{prop:iso_lattices} that all the vector bundles $E_{\R^n,\Lambda}$ are isomorphic regardless of the chosen lattice $\Lambda$. We will therefore concentrate on the simplest choice $\Lambda = \Z^n$.

The base space of the vector bundle $E = E_{\R^n,\Z^n}$ becomes $(\R^n / \Z^n) \times (\R^n / \Z^n) \cong \T^{2n}$. The fact that $E$ is a nontrivial vector bundle in this case is a consequence of the usual amalgam Balian--Low theorem in higher dimensions. The proof found in e.g.\ \cite[p.\ 164 Lemma 8.4.2]{groechenig} assumes that there is a non-vanishing continuous quasiperiodic function $F \colon \R^n \times \R^n \to \C$ and arrives at a contradiction. We will outline an alternate proof using Chern--Weil theory for vector bundles, see \cite{milnor}.

The base space $\T^{2n}$ has the structure of a smooth manifold, and a frame for the vector fields on $\T^{2n}$ is given by $\{ \partial_{x_j}, \partial_{\omega_j} \}_{j=1}^{n}$ where we view a coordinate in $\T^{2n}$ as $(x_1, \ldots, x_n, \omega_1, \ldots, \omega_n)$.

The vector bundle $E$ can be shown to be a smooth vector bundle over $\T^{2n}$. The smooth sections of $E$ can be identified with \emph{smooth} quasiperiodic functions $F \colon \R^n \times \R^n \to \C$.

A calculation shows that defining
\begin{align*}
\nabla_{\partial_{x_j}} F(x,\omega) &= \frac{\partial F}{\partial x_j}(x,\omega), \\
\nabla_{\partial_{\omega_j}} F(x,\omega) &= 2\pi i x_j F(x,\omega) + \frac{ \partial F}{\partial \omega_j}(x,\omega)
\end{align*}
for $(x,\omega) \in G \times \widehat{G}$, and extending linearly, defines a connection $\nabla$ on $E$. We have that
\begin{align*}
(\nabla_{\partial_{x_j}} \nabla_{\partial_{\omega_k}} - \nabla_{\partial_{\omega_k}} \nabla_{\partial_{x_j}}) F(x,\omega) &= \begin{cases}
       2\pi i F(x,\omega) &\quad \text{for $j=k$} \\
       0 &\quad \text{for $j \neq k$} \\
     \end{cases}
\end{align*}
for $1 \leq j,k \leq n$ and $(x,\omega) \in G \times \widehat{G}$. Hence the curvature $F^{\nabla}$ associated to $\nabla$ is given by
\[ F^{\nabla} = 2\pi i \sum_{j=1}^n \dee x_j \wedge \dee \omega_j .\]
It follows that the first Chern class of $E$ is given by
\[ c_1(E) = \frac{i}{2\pi} F^{\nabla} = - \sum_{j=1}^n \dee x_j \wedge \dee \omega_j .\]
This is a nontrivial element of the second cohomology group $H^2(\T^{2n},\Z) \cong \Z^{n(2n-1)}$ which is generated by all wedges of any two 1-forms $\dee x_j, \dee \omega_k$, $j,k=1,\ldots, n$. Since the line bundle has nontrivial first Chern class, it follows that the bundle is nontrivial.
\end{example}

\begin{proposition}\label{prop:compact_or_discrete}
Let $G$ be a second countable abelian group which is either compact or discrete, and let $\Lambda$ be any lattice in $G$. Then the vector bundle $E_{G,\Lambda}$ is trivial.
\end{proposition}

\begin{proof}
Suppose that $G$ is compact. Then $\Lambda$ is a compact and discrete subgroup of $G$, so it must be finite, say of order $r$. Since finite abelian groups are self-dual, $\widehat{G}/\Lambda^{\perp} \cong \widehat{\Lambda}$ must be finite of order $r$ as well. Thus, the base space $X_{G,\Lambda}$ of the vector bundle $E=E_{G,\Lambda}$ is homeomorphic to a disjoint union of $r$ copies of $G/\Lambda$:
\[ X_{G,\Lambda} = (G/\Lambda) \times (\widehat{G}/\Lambda^{\perp}) = \coprod_{k=1}^r (G/\Lambda) \times \{ \omega_k \} .\]
Here, $\omega_1, \ldots, \omega_r$ are coset representatives for $\Lambda^{\perp}$ in $\widehat{G}$. Now, the vector bundle $E$ is trivial if the restrictions of $E$ to each component $(G/\Lambda) \times \{ \omega_k \}$ is trivial. The restrictions are of the form
\[ E_k \coloneqq \pi^{-1}((G/\Lambda) \times \{ \omega_k \}) = \{ [x, \omega_k, z] : x \in G, z \in \C \} .\]
Define a map $\phi \colon E_k \to (G/\Lambda) \times \C$ from $E_k$ to the trivial bundle over $G/\Lambda$ of rank $1$ by
\[ \phi([x,\omega_k,z]) = ([x], \omega_k(x) z) .\]
This map is well-defined, as it maps $[x\lambda, \omega_k \tau, \overline{\omega_k(\lambda)} z]$ to $([x\lambda], \omega_k(x\lambda) \overline{\omega_k(\lambda)} z) = ([x], \omega_k(x) z)$ for $(x,\omega) \in G \times \widehat{G}$, $(\lambda,\tau) \in \Lambda \times \Lambda^{\perp}$ and $z \in \C$. It is clear that it is continuous and restricts to a linear isomorphism on each fiber, hence it is an isomorphism of vector bundles. This proves the assertion that $E$ is trivial.

If $G$ is discrete, then $G/\Lambda$ is both compact and discrete, hence finite. The argument above applies, with the roles of $(G,\Lambda)$ and $(\widehat{G},\Lambda^{\perp})$ interchanged.
\end{proof}

At this point, we would like to point out that even though one does not have a Balian--Low theorem for the Feichtinger algebra in compact or discrete groups, this does not mean there there does not exist Balian-Low phenomena at all in these groups. In fact, it is shown in \cite{olsen_finite} that there exists a Balian--Low theorem in finite cyclic groups.

%
%
%
%
%

The following result follows directly from combining the main result of \cite{kankut} with \Cref{thm:bundle_bal}.

\begin{proposition}
Let $G$ be a second countable, locally compact abelian group that is compactly generated and has noncompact connected component of the identity. Then for any lattice $\Lambda$ in $G$, the vector bundle $E_{G,\Lambda}$ is nontrivial.
\end{proposition}

\section{The group \texorpdfstring{$\R \times \Q_p$}{R x Qp}}\label{sec:rp}

In this section, we will show that when $G = \R \times \Q_p$ and $\Lambda$ is a certain lattice in $G$, then the associated bundle $E_{G,\Lambda}$ is nontrivial. By \Cref{thm:bundle_bal}, this will prove \Cref{problem:bl} for this setting, thereby giving us a Balian--Low theorem in this new setting.\\

Recall that the locally compact field $\Q_p$ of \emph{$p$-adic numbers} is the completion of $\Q$ with respect to the absolute value $| x |_p = p^{-k}$ where $x = p^k (a/b)$ for $a,b,k \in \Z$, $p \nmid a$, $p \nmid b$ and $|0|_p = 0$. Every $p$-adic number $x$ can be expressed uniquely as a converging series
\[ x = \sum_{k=K}^\infty x_k p^k \]
where $K$ is a (possibly negative) integer and $x_k \in \{ 0, \ldots, p-1 \}$ for every integer $k \geq K$. The ring of \emph{$p$-adic integers} $\Z_p$ is the compact subgroup of $p$-adic numbers $x$ for which $|x|_p \leq 1$, or, equivalently, the series of the form $\sum_{k=0}^\infty x_k p^k$.

The group $\Q_p$ is self-dual via the pairing $\Q_p \times \Q_p \to \T$, $(x,y) \mapsto e^{2\pi i \{ xy \}_p}$, where
\[ \left\{ \sum_{k=K}^\infty x_k p^k \right\}_p = \sum_{k=K}^{-1} x_k p^k .\]
We identify the group $\widehat{\R \times \Q_p}$ with itself via $((s,x),(t,y)) \mapsto e^{2\pi i st}e^{-2\pi i \{ xy \}_p}$ (note the sign). A remarkable fact is that
\[ \Z[1/p] = \{ a/p^k : a,k \in \Z \} \]
can be embedded diagonally into $\R \times \Q_p$ as a lattice. In other words, the set $\Z[1/p]^{\Delta} = \{ (q,q) : q \in \Z[1/p] \}$ is a lattice in $\R \times \Q_p$ (analogous to \cite[Theorem 5.11]{ramakrishnan}). Under the identification of the Pontryagin dual of $\R \times \Q_p$ above, the annihilator of $\Z[1/p]^{\Delta}$ is identified with $\Z[1/p]$. Thus we have that
\begin{equation} \widehat{\Z[1/p]} \cong \widehat{\Z[1/p]^{\Delta}} \cong \frac{\R \times \Q_p}{(\Z[1/p]^{\Delta})^{\perp}} \cong \frac{\R \times \Q_p}{\Z[1/p]^{\Delta}} . \label{eq:sol_iso}
\end{equation}

Now $\Z[1/p]$ can be realized as the direct limit of the sequence
\[ \Z \xrightarrow{\cdot p} \Z \xrightarrow{\cdot p} \Z \xrightarrow{\cdot p} \cdots  \]
where each map is multiplication by $p$. We index this sequence by the natural numbers including zero. The injection of the $k$-th copy of $\Z$ into $\Z[1/p]$ is given by $a \mapsto a/p^k$. Taking Pontraygin duals, we obtain that $\widehat{\Z[1/p]}$ is isomorphic to the inverse limit of the sequence
\[ \cdots \xrightarrow{\widehat{\cdot p}} \widehat{\Z} \xrightarrow{\widehat{\cdot p}} \widehat{\Z} \xrightarrow{\widehat{\cdot p}} \widehat{\Z} . \]
Using the usual identification $\widehat{\Z} \cong \R / \Z$, we obtain that $\widehat{\Z[1/p]}$ is isomorphic to the inverse limit of
\[ \cdots \xrightarrow{\cdot p} \R / \Z \xrightarrow{\cdot p} \R / \Z \xrightarrow{\cdot p} \R / \Z  \]
where each map is multiplication by $p$. This space is known as the \emph{$p$-solenoid} and is denoted by  $\sol{p}$. It is an example of a compact Hausdorff space that is connected, but not path-connected.

Denote by $\pi_k \colon \widehat{\Z[1/p]} \to \R / \Z$ the projection down to the $k$-th copy of $\R/\Z$. It is given by mapping a character $\omega \in \widehat{\Z[1/p]}$ to the class $[t] \in \R / \Z$ such that $\omega(a/p^k) = e^{2\pi i at}$ for all $a,k \in \Z$.

Now using \eqref{eq:sol_iso}, $\omega$ is obtained from a class $[s,x] \in (\R \times \Q_p)/\Z[1/p]^{\Delta}$, i.e.\ it is given by $\omega(q) = e^{2\pi i q s}e^{-2\pi i \{ q x \}_p}$ for $q \in \Z[1/p]$. Letting $\pi_k(\omega) = [t]$, we have that
\[ e^{2\pi a t} = \omega(a/p^k) = e^{2\pi i a s/p^k} e^{-2\pi i \{ a x / p^k \}_p} \]
for all $a,k \in \Z$, so that $[t] = [ p^{-k} s - \{ p^{-k} x \}_p ]$. We thus have the following proposition:

\begin{proposition}\label{prop:solenoid_limit}
The $p$-solenoid $\sol{p}$ given by the inverse limit of the sequence
\[ \cdots \xrightarrow{\cdot p} \R / \Z \xrightarrow{\cdot p} \R / \Z \xrightarrow{\cdot p} \R / \Z \]
is isomorphic to the quotient group $(\R \times \Q_p) / \Z[1/p]^{\Delta}$. The projection $\pi_k$ from $(\R \times \Q_p) / \Z[1/p]^{\Delta}$ down to the $k$-th factor of $\R / \Z$ is given by
\[ \pi_k([s,x]) = [p^{-k} s - \{ p^{-k} x \}_p ] \]
for $(s,x) \in R \times \Q_p$.
\end{proposition}

We now investigate the line bundle $\tilde{E} = E_{G,\Lambda}$ in the case of $G = \R \times \Q_p$ and $\Lambda = \Z[1/p]^{\Delta}$. The base space becomes
\[ X_{G,\Lambda} = \frac{ \R \times \Q_p}{\Z[1/p]^{\Delta}} \times \frac{ \widehat{\R \times \Q_p}}{(\Z[1/p]^{\Delta})^{\perp}} \cong \left( \frac{\R \times \Q_p}{\Z[1/p]^{\Delta}} \right)^2 \]
which is homeomorphic to $\sol{p}^2$ by \Cref{prop:solenoid_limit}. Thus, the vector bundle $\tilde{E}$ is a line bundle over $\sol{p}^2$, and by \Cref{prop:solenoid_limit}, $\sol{p}^2$ can be realized as the inverse limit of the sequence
\[ \cdots \rightarrow (\R / \Z)^2 \xrightarrow{(\cdot p, \cdot p)} (\R / \Z)^2 \xrightarrow{(\cdot p, \cdot p)} (\R / \Z)^2 . \]
Here $(\cdot p, \cdot p)([s],[t]) = ([ps],[pt])$ for $([s],[t]) \in (\R/\Z)^2$, and the projection $\pi_k^2 \colon \sol{p}^2 \to (\R/\Z)^2$ given by
\begin{equation}
    \pi_k^2([s,x],[t,y]) = ([p^{-k} s - \{ p^{-k} x \}_p], [p^{-k}t - \{ p^{-k} y \}_p ]) \label{eq:projection2}
\end{equation}
for $([s,x],[t,y]) \in \sol{p}^2$.

Key to understanding vector bundles over inverse limits of spaces is the following result. It can be seen by combining \cite[Example 5.2.4]{blackadark} with the correspondence between vector bundles and finitely generated projective modules, as well as the continuity of the functor $C(\cdot)$ from compact Hausdorff spaces to unital $C^*$-algebras:

\begin{proposition}\label{prop:inverse_limit}
Suppose $X$ is the inverse limit of the sequence of compact Hausdorff spaces
\[ \cdots \xrightarrow{f_3} X_3 \xrightarrow{f_2} X_2 \xrightarrow{f_1} X_1  .\]
Then the monoid $\Vect(X)$ of isomorphism classes of vector bundles over $X$ (under direct sum) is isomorphic to the direct limit of the monoids
\[ \Vect(X_1) \xrightarrow{f_1^*} \Vect(X_2) \xrightarrow{f_2^*} \Vect(X_3) \xrightarrow{f_3^*} \cdots , \]
where the induced maps are given by pullback of vector bundles.
\end{proposition}

The identity in $\Vect(X)$ is the class of the unique trivial bundle of rank $0$. Removing this class from $\Vect(X)$, i.e.\ considering only bundles of positive rank, we obtain a semigroup that we will denote by $\Vect^+(X)$. The functor $\Vect^+$ preserves inverse limits in the same sense that $\Vect$ does in \Cref{prop:inverse_limit}.

Denote by $\Z^+$ the set of positive integers. By \cite[Theorem 3.9]{cancellation} and \cite[Proposition 4.1]{projmultres1}, the semigroup $\Vect^+(\T^2)$ is isomorphic to $\Z^+ \times \Z$, where $(q,a)$ represents the vector bundle $E_{q,a}$ over $\T^2$ of rank $q$ and Chern class $-a \in H^2(\T^2,\Z) \cong \Z$. The vector bundle $E = E_{\R,\Z}$ corresponds to $(q,a)=(1,-1)$ by \Cref{ex:realbundle}. It is easily shown that $(\cdot p, \cdot p)^*(E_{q,a}) \cong E_{q,p^2a}$. Hence, the induced map of $(\cdot p, \cdot p) \colon (\R/\Z)^2 \to (\R/\Z)^2$ on the level of vector bundles is given by $(q,a) \mapsto (q,p^2a)$. This gives us that $\Vect^+(\sol{p}^2)$ is the direct limit of the sequence of monoids
\[ \Z^+ \times \Z \xrightarrow{(\cdot 1,\cdot p^2)} \Z^+ \times \Z \xrightarrow{(\cdot 1, \cdot p^2)} \Z^+ \times \Z \xrightarrow{(\cdot 1, \cdot p^2)} \cdots \]
In other words, $\Vect^+(\sol{p}^2) \cong\Z^+ \times \Z[1/p]$. It follows that we obtain the following classification of vector bundles over $\sol{p}^2$:

\begin{proposition}\label{prop:vectsol}
Let $E_{q,a}$ denote the vector bundle of rank $q$ and Chern class $-a$ over $\T^2$. Then every vector bundle over $\sol{p}^2$ of positive rank is of the form $(\pi_k^2)^*(E_{q,a})$ for some $k,q \in \Z^+$ and $a \in \Z$, where $\pi_k^2$ is the map in \eqref{eq:projection2}. Moreover, $(\pi_k^2)^*(E_{q,a}) \cong (\pi_{k'}^2)^*(E_{q',a'})$ if and only if $q=q'$ and $a/p^{2k} = a'/p^{2k'}$. Thus, $(\pi_k^2)^*(E_{q,a})$ is trivial if and only if $a = 0$.
\end{proposition}

The next proposition, namely \Cref{prop:pullback_bundle}, determines the bundle $E_{\R \times \Q_p,\Z[1/p]^{\Delta}} \to \sol{p}^2$ among the bundles in \Cref{prop:vectsol}. Before the proof, we present the functoriality of the section functor.

To begin with, let $A$ and $B$ be unital $C^*$-algebras and let $\phi \colon A \to B$ be a $*$-homomorphism. If $\mathcal{E}$ is a left finitely generated projective $A$-module, there is a way to construct a left finitely generated projective $B$-module $\phi_*(\mathcal{E})$ known as \emph{extension of scalars} \cite[p.\ 60]{Va01}. In terms of projection matrices, extension of scalars can be described as follows: If $\mathcal{E}$ is represented by the projection $P=(p_{i,j})_{i,j=1}^k \in M_k(A)$, i.e.\ $\mathcal{E} \cong A^k P$, then we have that
\begin{equation}
    \phi_*(A^k P) \cong B^k \phi_*(P) . \label{eq:induced_proj}
\end{equation}
Here, $\phi_*(P)$ is the matrix $(\phi(p_{i,j}))_{i,j=1}^k \in M_k(B)$.

Now suppose we are in the following situation: We have a continuous map $f \colon X \to Y$ of compact Hausdorff spaces. We then get an induced $*$-homomorphism $C(f) \colon C(Y) \to C(X)$ given by precomposition by $f$. We can consider the extension of scalars along this $*$-homomorphism. Let $E \to Y$ be a vector bundle over $Y$. Then we can form the pullback $f^*(E)$ of $E$ along $f$, which is a vector bundle over $X$. Now $\Gamma(E)$ is a left finitely generated projective $C(Y)$-module, while $\Gamma(f^*(E))$ is a finitely generated projective $C(X)$-module. The section functor is functorial in the sense that if one performs extension of scalars on $\Gamma(E)$ using the induced map $C(f)$, then the resulting $C(Y)$-module is isomorphic to $\Gamma(f^*(E))$. That is, we have that
\begin{equation}
C(f)_*(\Gamma(E)) \cong \Gamma(f^*(E)) . \label{eq:extension_of_scalars}
\end{equation}
For a proof of $\eqref{eq:extension_of_scalars}$, see \cite[Proposition 2.12]{Va01}.

The proof of the following proposition will employ fundamental domains, which are introduced in \Cref{appendix:fundoms}.

\begin{proposition}\label{prop:pullback_bundle}
Let $E$ denote the line bundle $E_{\R,\Z}$ over $(\R/\Z)^2$ and let $\tilde{E}$ denote the line bundle $E_{\R \times \Q_p, \Z[1/p]^{\Delta}}$ over $\sol{p}^2$, both defined as in \Cref{sec:bundle}. Then $\tilde{E}$ is isomorphic to the pullback of $E$ along $\pi_0^2 \colon \sol{p}^2 \to (\R/\Z)^2$, where $\pi_0^2$ is given as in \eqref{eq:projection2}.
\end{proposition}

\begin{proof}
Our proof goes as follows: Instead of showing directly that $\tilde{E} \cong (\pi_0^2)^*(E)$, we can instead show (by \Cref{prop:serre_swan}) that $\Gamma(\tilde{E}) \cong \Gamma((\pi_0^2)^*(E))$ as $C(\sol{p}^2)$-modules.

We consider the continuous map $\pi_0^2 \colon \sol{p}^2 \to (\R/\Z)^2$, which induces the map
\[ C(\pi_0^2) \colon C((\R/\Z)^2) \to C(\sol{p}^2) \]
given by precomposition by $\pi_0^2$.

We have that $\Gamma(\tilde{E}) \cong \sec{\tilde{G}}{\tilde{\Lambda}}$ by \Cref{prop:hilbert_sec}. We also have that
\begin{align*}
    \Gamma((\pi_0^2)^*(E)) & \cong C(\pi_0^2)_*(\Gamma(E)) && \text{by $\eqref{eq:extension_of_scalars}$} \\
    & \cong C(\pi_0^2)_*(\sec{\R}{\Z}) && \text{by \Cref{prop:hilbert_sec}} .
\end{align*}
The goal is thus to show that $\sec{\tilde{G}}{\tilde{\Lambda}}$ and $C(\pi_0^2)_*(\sec{\R}{\Z})$ are isomorphic as $C(\sol{p}^2)$-modules. To show this, it suffices to show that they can be represented by the same projection matrix over $C(\sol{p}^2)$, i.e.\ that $\sec{\tilde{G}}{\tilde{\Lambda}} \cong A^k P \cong C(\pi_0^2)_*(\sec{\R}{\Z})$ for some $k \in \N$ and $P \in M_k(C(\sol{p}^2))$. This will be our approach.

We begin by constructing normalized tight module frames for the modules $\sec{\tilde{G}}{\tilde{\Lambda}}$ and $C(\pi_0^2)_*(\sec{\R}{\Z})$, using fundamental domains from \Cref{appendix:fundoms}.

Set $\tilde{G} = \R \times \Q_p$ and $\tilde{\Lambda} = \Z[1/p]^{\Delta}$. A fundamental domain for $\tilde{\Lambda}$ in $\tilde{G}$ is given by $\tilde{B}= [0,1) \times \Z_p$. To see this, note that if $(s,x) \in \R \times \Q_p$, then
\begin{equation}
(s,x) = (\{ x \}_p + \lfloor s - \{ x \}_p \rfloor, \{ x \}_p + \lfloor s - \{ x \}_p \rfloor ) + ( \text{fr}( s - \{ x \}_p ), x - \{ x \}_p - \lfloor s - \{ x \}_p \rfloor ) \label{eq:rqp_fundom}
\end{equation}
where $\lfloor  \cdot \rfloor$ denotes the floor function $\R$ and $\text{fr}$ denotes the fractional part of a real number, i.e.\ $\text{fr}(s) = s - \lfloor s \rfloor$. The first pair in \eqref{eq:rqp_fundom} is in $\Z[1/p]^{\Delta}$ while the second pair is in $[0,1) \times \Z_p$. Thus, the generalized floor function $\lfloor  \cdot \rfloor_{\tilde{B}} \colon \R \times \Q_p \to \Z[1/p]$ from \eqref{eq:floor_prop} associated to the fundamental domain $\tilde{B}= [0,1) \times \Z_p$ is given by
\begin{equation}
\lfloor (s,x) \rfloor_{\tilde{B}} = \{ x \}_p + \lfloor s - \{ x \}_p \rfloor  \label{eq:rqp_floor}
\end{equation}
for $(s,x) \in \R \times \Q_p$. Now since $\Z_p$ is clopen in $\Q_p$ we have that
\begin{align}
\partial ( [0,1) \times \Z_p ) &= \partial([0,1)) \times \overline{\Z_p} \cup \overline{[0,1)} \times \partial(\Z_p) = \{ 0, 1 \} \times \Z_p.
\end{align}
By \Cref{prop:loc_const}, $\lfloor \cdot \rfloor_{\tilde{B}}$ is discontinuous at $(s,x)$ if and only if $(s,x) \in \partial( [0,1) \times \Z_p) = \{ 0,1 \} \times \Z_p$. From \eqref{eq:rqp_floor}, this happens if and only if $s - \{ x \}_p \in \Z$.\\

By \Cref{ex:qpfilter}, we can find a continuous, $\Lambda$-periodic function $f \colon \R \to \C$ with $f(0) = 0$ and
\[ |f(s)|^2 + |f(s+1/2)|^2 = 1 \]
for all $s \in \R$. Consequently, the set $\{ G_1, G_2 \}$ where
\begin{align*}
G_1(s,t) &= f(s) J_{[0,1)}(s,t) = f(s) e^{-2\pi i \lfloor s \rfloor t}, \\
G_2(s,t) &= f(s+1/2) J_{[0,1)}(s+1/2,t) = e^{-2\pi i \lfloor s+1/2 \rfloor t},
\end{align*}
is a normalized tight module frame for $\sec{\R}{\Z}$ as in \Cref{cor:specific_frame_qp}. By \Cref{prop:frame_proj}, the matrix $P \in M_2(C(\T^2))$ given by
\begin{align*}
P([s],[t]) &= \begin{pmatrix} |G_1(s,t)|^2 & G_1(s,t)\overline{G_2(s,t)} \\[1ex] G_2(s,t)\overline{G_1(s,t)} & |G_2(s,t)|^2 \end{pmatrix}
\end{align*}
represents $\sec{\R}{\Z}$, i.e.\ $\sec{\R}{\Z} \cong C(\T^2)^2 P$. Now it follows from \eqref{eq:induced_proj} that the module $(\pi_0^2)_*(\sec{\R}{\Z}) = (\pi_0^2)_*(C(\T^2)^2P)$ is represented by the projection matrix $Q = (\pi_0^2)_*(P) = P \circ \pi_0^2$, which is given by $Q = (Q_{i,j})_{i,j=1}^2$ where
\begin{align*}
Q_{1,1}([s,x],[t,y]) &= |G_1(s-\{x\}_p, t-\{y\}_p)|^2, \\
Q_{1,2}([s,x],[t,y]) &= G_1(s-\{x\}_p,t-\{y\}_p)\overline{G_2(s-\{x\}_p,t-\{y\}_p)}, \\
Q_{2,1}([s,x],[t,y]) &= G_2(s-\{x\}_p,t-\{y\}_p) \overline{G_1(s-\{x\}_p,t-\{y\}_p)}, \\
Q_{2,2}([s,x],[t,y]) &= |G_2(s-\{x\}_p,t-\{y\}_p)|^2 .
\end{align*}

Next, note that with $x_1 = (0,0)$ and $x_2 = (-1/2,0)$, the sets $x_1 + \tilde{B}^{\circ} + \tilde{\Lambda} = (0,1) \times \Z_p + \Z[1/p]^{\Delta}$ and $x_2 + \tilde{B}^{\circ} + \tilde{\Lambda} = (-1/2,1/2) \times \Z_p + \Z[1/p]^{\Delta}$ form an open cover of $\R \times \Q_p$ as in \Cref{cor:specific_frame_qp}. Now define $\tilde{f} \colon \R \times \Q_p \to \C$ by $\tilde{f}(s,x) = f(s - \{ x \}_p)$. Then $\tilde{f}$ is a continuous, $\Z[1/p]^{\Delta}$-periodic function and
\[ |\tilde{f}(s,x)|^2 + |\tilde{f}(s+1/2,x)|^2 = |f(s-\{ x \}_p)|^2 + |f(s-\{ x \}_p + 1/2)|^2 = 1 \]
for all $(s,x) \in \R \times \Q_p$. Moreover, if $(k,x) \in \partial([0,1) \times \Z_p) = \{ 0,1 \} \times \Z_p$ then
\[ \tilde{f}(k,x) = f(k - \{ x \}_p) = f(k) = 0, \]
since $f$ is zero on integers. This shows that $\tilde{f}$ satisfies the requirements in \Cref{cor:specific_frame_qp} with respect to the fundamental domain $\tilde{B}$ and $x_1, x_2$. Consequently, with
\begin{align*}
\tilde{G}_1(s,x,t,y) &= \tilde{f}(s,x) J_{\tilde{B}}(s,x,t,y) \\
\tilde{G}_2(s,x,t,y) &= \tilde{f}(s+1/2,x) J_{\tilde{B}}(s+1/2,x,t,y) ,
\end{align*}
we have that $\{ \tilde{G}_1, \tilde{G}_2 \}$ is a normalized tight frame for $\sec{\tilde{G}}{\tilde{\Lambda}}$. By \Cref{prop:frame_proj}, it follows that the matrix $\tilde{P} = (\tilde{P}_{i,j})_{i,j} \in M_2(C(\sol{p}^2))$ given by
\[ \tilde{P}([s,x],[t,y]) = \begin{pmatrix} |\tilde{G}_1(s,x,t,y)|^2 & \tilde{G}_1(s,x,t,y), \tilde{G}_2(s,x,t,y) \\[1ex] \tilde{G}_2(s,x,t,y) \tilde{G}_1(s,x,t,y) & |\tilde{G}_2(s,x,t,y)|^2 \end{pmatrix} \]
represents the module $\sec{\tilde{G}}{\tilde{\Lambda}}$, i.e.\ $\sec{\tilde{G}}{\tilde{\Lambda}} \cong C(\sol{p}^2)^2 \tilde{P}$. We now compute $\tilde{P}$ and show that it is equal to the matrix $Q$. First, note that
\[ |\tilde{G}_1(s,x,t,y)|^2 = |\tilde{f}(s,x)|^2 = |f(s-\{x\}_p)|^2 = |G_1(s-\{x\}_p, t-\{y\}_p)|^2 \]
and similarly, $|\tilde{G}_2(s,x,t,y)|^2 = |G_2(s-\{x\}_p, t-\{y\}_p)|^2$. Thus, $\tilde{P}_{1,1} = Q_{1,1}$ and $\tilde{P}_{2,2} = Q_{2,2}$. Next, note that using \eqref{eq:rqp_floor}, $J_{\tilde{B}}$ is given by
\[ J_{\tilde{B}}(s,x,t,y) = e^{-2\pi i ( \{ x \}_p + \lfloor s - \{ x \}_p \rfloor ) t} e^{2\pi i \{ ( \{ x \}_p + \lfloor s - \{ x\}_p \rfloor ) y \}_p} \]
for $(s,x),(t,y) \in \R \times \Q_p$. Thus, calculating $J_{\tilde{B}}(s,x,t,y) \overline{ J_{\tilde{B}}(s+1/2,x,t,y)}$, the exponentials involving $\{x\}_p$ cancel and we are left with
\begin{align*}
J_{\tilde{B}}(s,x,t,y) \overline{ J_{\tilde{B}}(s+1/2,x,t,y)} &= e^{-2\pi i (\lfloor s - \{ x\}_p \rfloor - \lfloor s + 1/2 - \{ x\}_p \rfloor ) t} \\
& \qquad \times e^{2\pi i \{ \lfloor s - \{ x \}_p \rfloor - \lfloor s + 1/2 - \{x \}_p \rfloor ) y \}_p} \\
&= e^{-2\pi i ( \lfloor s + \{ x \}_p \rfloor - \lfloor s + 1/2 - \{ x \}_p \rfloor ) (t - \{ y \}_p ) } .
\end{align*}
It follows that
\begin{align*}
\tilde{P}_{1,2}([s,x],[t,y]) &= \tilde{G}_1(s,x,t,y) \overline{\tilde{G}_2(s,x,t,y)} \\
&= \tilde{f}(s,x) \overline{ \tilde{f}(s+1/2,x) } J_{\tilde{B}}(s,x,t,y) \overline{ J_{\tilde{B}}(s+1/2,x,t,y)} \\
&= f(s-\{x\}_p) \overline{ f( s + 1/2 - \{x \}_p )} e^{-2\pi i ( \lfloor s + \{ x \}_p \rfloor - \lfloor s + 1/2 - \{ x \}_p \rfloor ) (t - \{ y \}_p ) } \\
&= f(\pi_0(s,x)) \overline{ f(\pi_0(s+1/2,x)} J_B(s-\{x\}_p, t-\{y\}_p) \\
& \qquad \times \overline{ J_B( s+1/2-\{x\}_p, t - \{ y\}_p)} \\
&= G_1(s-\{x\}_p, t-\{y\}_p) \overline{ G_2(s-\{x\}_p, t - \{y\}_p)} \\
&= Q_{1,2}([s,x],[t,y]).
\end{align*}
This implies that
\[ \tilde{P}_{2,1} = \overline{ \tilde{P}_{1,2}} = \overline {Q_{1,2}} = Q_{2,1} .\]
We have now shown that $\tilde{P}_{i,j} = Q_{i,j}$ for all $i,j=1,2$. Thus $Q = \tilde{P}$, which finishes the proof.
\end{proof}

With \Cref{prop:pullback_bundle} proved, our final main result is a straightforward corollary:

\begin{theorem}\label{thm:rqp}
Let $G = \R \times \Q_p$ and let $\Lambda = \Z[1/p]^{\Delta}$. Then the line bundle $\tilde{E} = E_{G,\Lambda}$ over $\sol{p}^2$ is nontrivial. Hence \Cref{problem:bl} holds for $(G,\Lambda)$, that is, if $\eta \in S_0(\R \times \Q_p)$, then the Gabor system
\[ \mathcal{G}(\eta, \Lambda \times \Lambda^{\perp}) = \{ (s,x) \mapsto e^{2\pi i rs} e^{-2\pi i \{ r x \}_p} \eta( s - q, x - q) : q,r \in \Z[1/p] \} \]
is not a frame for $L^2(\R \times \Q_p)$.
\end{theorem}

\begin{proof}
By \Cref{prop:pullback_bundle}, $\tilde{E} \cong \pi_0^*(E_{1,1})$. By \Cref{prop:vectsol}, $\tilde{E}$ is then nontrivial, since otherwise it would be of the form $(\pi_0^2)^*(E_{q,0})$ for some $q \in \N$. By \Cref{thm:bundle_bal}, the nontriviality of $\tilde{E}$ implies \Cref{problem:bl} for $(G,\Lambda)$.
\end{proof}

\appendix

\section{Fundamental domains}\label{appendix:fundoms}

In this appendix, we collect some results on fundamental domains and the construction of specific module frames for $\sec{G}{\Lambda}$. They are used in the proof of \Cref{prop:pullback_bundle} and in the proof of \Cref{prop:feichtinger_zak_cont_app}.

Let $\Lambda$ be a lattice in a second countable LCA group $G$. A measurable set $B \subseteq G$ is called a \emph{fundamental domain} (or \emph{Borel section}) \cite{mackey1} for $\Lambda$ in $G$ if the collection $\{ \lambda B : \lambda \in \Lambda \}$ forms a partition of $G$. Equivalently, every $x \in G$ can be written uniquely as $x=b\lambda$ where $b \in B$ and $\lambda \in \Lambda$.

\begin{proposition}\label{prop:fund_int}
Let $\Lambda$ be a lattice in a second countable, locally compact abelian group $G$. Then there exists a relatively compact fundamental domain $B$ for $\Lambda$ in $G$ with nonempty interior.
\end{proposition}

\begin{proof}
Let $\mu$ denote a Haar measure on $G$. Since $G$ is both locally compact and second countable, it follows that $G$ is $\sigma$-compact. By \cite[p.\ 35, paragraph (0.40) 1)]{margulis}, there exists a relatively compact fundamental domain $B$ for $\Lambda$ in $G$ with $\mu(\partial B) = 0$.

We claim that $B$ cannot be nowhere dense. Suppose otherwise, i.e.\ $(\overline{B})^{\circ} = \emptyset$. Then $(\overline{\lambda B})^{\circ} = \lambda (\overline{B})^{\circ} = \emptyset$ for every $\lambda \in \Lambda$. Since $\Lambda$ is discrete in $G$ and $G$ is second countable, $\Lambda$ has to be countable. The identity
\[ G = \bigcup_{\lambda \in \Lambda} \lambda B \]
then expresses $G$, a locally compact Hausdorff space, as a countable union of nowhere dense subsets. This contradicts the Baire category theorem, so we must conclude that $(\overline{B})^{\circ} \neq \emptyset$.

Let $U$ be a nonempty open subset of $\overline{B}$. Now suppose for a contradiction that $B^{\circ} = \emptyset$. Then $\overline{B} = B^{\circ} \cup \partial B = \partial B$, so $\mu(U) \leq \mu(\overline{B}) = \mu(\partial B) = 0$, which contradicts the properties of Haar measure. Hence $B^{\circ} \neq \emptyset$.
\end{proof}



Let $B$ be a fundamental domain for the lattice $\Lambda$ in $G$. We define the function $\lfloor \cdot \rfloor_B \colon G \to \Lambda$ by letting $\lfloor x \rfloor_B$ be the unique $\lambda \in \Lambda$ such that $x = b\lambda$ for some (unique) $b \in B$. We think of this as a generalized floor function.

Note that we have
\begin{equation}
\lfloor x \lambda \rfloor_B = \lfloor x \rfloor_B \lambda \label{eq:floor_prop}
\end{equation}
for all $x \in G$ and $\lambda \in \Lambda$.

\begin{proposition}\label{prop:loc_const}
The function $\lfloor \cdot \rfloor_B$ is locally constant on $B^{\circ} \Lambda$ and discontinuous on $(B^{\circ} \Lambda)^c = (\partial B)\Lambda$.
\end{proposition}

\begin{proof}
We begin by showing that $\lfloor \cdot \rfloor_B$ is locally constant on $B^{\circ} \Lambda$. Suppose $x \in B^{\circ} \Lambda$. Then $x = b \lambda$ for $b \in B$, $\lambda \in \Lambda$, and there exists an open set $U \subseteq B$ with $b \in U$. Then $x \in \lambda U$. If $y \in \lambda U$, say $y = b' \lambda$ with $b' \in U$, then $\lfloor y \rfloor_B = \lambda = \lfloor x \rfloor_B$, which proves that $\lfloor \cdot \rfloor_B$ is constant on the set $\lambda U$.

Next, we show that if $\lfloor \cdot \rfloor_B$ is continuous at $x \in G$, then $x \in B^{\circ} \Lambda$. Suppose $\lfloor \cdot \rfloor_B$ is continuous at $x = b \lambda$ for $b \in B$, $\lambda \in \Lambda$. Let $V$ be a neighbourhood of $\lambda$ such that $V \cap \Lambda = \{ \lambda \}$. By continuity, there exists a neighbourhood $U$ of $x$ for which $\lfloor \cdot \rfloor_B(U) \subseteq V$. But since $\lfloor \cdot \rfloor_B$ maps into $\Lambda$, we have that $\lfloor y \rfloor_B = \lambda$ for all $y \in U$. Hence $\lfloor z \rfloor_B = 1$ for all $z \in \lambda^{-1}U$, i.e.\ $\lambda^{-1}U \subseteq B$. Since $b \in \lambda^{-1}U$ we then have that $b \in B^{\circ}$. Thus $x \in B^{\circ} \Lambda$.
\end{proof}

We describe how to obtain a specific module frame for $\sec{G}{\Lambda}$ from certain fundamental domains. This is inspired by \cite[p.\ 457]{projmultres1}. Given a fundamental domain $B$ for $\Lambda$ in $G$, define the function $J_B \colon G \times \widehat{G} \to \C$ by
\begin{equation}
J_B(x,\omega) = \overline{ \omega( \lfloor x \rfloor_B) }.
\end{equation}
Note that for $(x,\omega) \in G \times \widehat{G}$ and $(\lambda,\tau) \in \Lambda \times \Lambda^{\perp}$, we have that
\[ J_B(x\lambda,\omega\tau) = \overline{(\omega\tau)(\lfloor x \lambda \rfloor_B)} = \overline{ \omega( \lfloor x \rfloor_B \lambda)} = \overline{\omega(\lambda)} J_B(x,\omega) . \]
Hence $J_B$ satisfies the quasiperiodicity relation \eqref{eq:quasi}. However, it is not necessarily continuous everywhere on $G \times \widehat{G}$, so it is not an element of $\sec{G}{\Lambda}$ in general. But by \Cref{prop:loc_const}, $J_B$ is continuous on $B^{\circ} \Lambda \times \widehat{G}$. Thus, if $f \colon G \to \C$ is a continuous, $\Lambda$-periodic function with $f|_{\partial B} = 0$, then the function $G \colon G \times \widehat{G} \to \C$ given by
\begin{equation}
G(x,\omega) = f(x) J_B(x,\omega) \label{eq:Gfunction}
\end{equation}
for $(x,\omega) \in G \times \widehat{G}$ is continuous and quasiperiodic, hence an element of $\sec{G}{\Lambda}$. We will use these functions in the following proposition to construct module frames for $\sec{G}{\Lambda}$:

\begin{proposition}\label{cor:specific_frame_qp}
Let $B$ be a relatively compact fundamental domain with nonempty interior and let $x_1, \ldots, x_k \in G$ be such that $\{ x_j B^{\circ} \Lambda \}_{j=1}^k$ covers $G$. Let $f \in C(G)$ be a $\Lambda$-periodic function with $f|_{\partial B} = 0$ that satisfies
\begin{equation}
\sum_{j=1}^k |f(x_j^{-1}x)|^2 = 1 \label{eq:filter_cor}
\end{equation}
for all $x \in G$. Set
\begin{equation}
G_j(x,\omega) = f(x_j^{-1}x) J_B(x_j^{-1}x,\omega).
\end{equation}
Then is a normalized tight module frame for $\sec{G}{\Lambda}$.
\end{proposition}

\begin{proof}
As described in the discussion preceding the proposition, the function $f(x) J_B(x,\omega)$ is continuous on the whole of $G \times \widehat{G}$ and hence an element of $\sec{G}{\Lambda}$. By translation, the functions $f(x_j^{-1}x) J_B(x_j^{-1}x,\omega)$ are also elements of $\sec{G}{\Lambda}$ for each $1 \leq j \leq k$. Finally, we have that
\[ \sum_{j=1}^k |G_j(x,\omega)|^2 = \sum_{j=1}^k |f(x_j^{-1}x)|^2 = 1 \]
for all $x \in G$. Hence, by \Cref{prop:frame_in_qp}, the set $\{ G_1, \ldots, G_k \}$ is a normalized tight module frame for $\sec{G}{\Lambda}$.
\end{proof}

\begin{example}\label{ex:qpfilter}
In this example we show how wavelet filter functions give us module frames as in \Cref{cor:specific_frame_qp} in the case $G = \R$, $\Lambda = \Z$. The example closely follows the discussion in \cite[p. 457]{projmultres1}.

Let $m \colon \R \to \C$ be a continuous wavelet filter function for dilation by 2 (see for instance \cite[Section 12.6]{heilbases}), i.e.\ it is $\Z$-periodic and satisfies
\begin{equation}
|m(s)|^2 + |m(s+1/2)|^2 = 1 \label{eq:filterdil2}
\end{equation}
for all $s \in \R$. Suppose also that $m$ is normalized so that $m(0)=1$. Then by \eqref{eq:filterdil2} we have that $m(1/2) = 0$, which gives $m(k+1/2) = 0$ for all $k \in \Z$ by periodicity. Let $B$ be the fundamental domain $[0,1)$ for $\Z$ in $\R$, and set $x_1 = 0$, $x_2 = -1/2$. We see that $f$ given by $f(s) = m(s-1/2)$ is continuous, vanishes on $\Z = (\partial B)\Lambda$ and satisfies \eqref{eq:filter_cor}, so by \Cref{cor:specific_frame_qp}, the functions
\begin{align*}
G_1(s,t) &= f(s) J_{[0,1)}(s,t), \\
G_2(s,t) &= f(s+1/2) J_{1/2+[0,1)}(s,t)
\end{align*}
form a normalized tight frame for $\sec{\R}{\Z}$.

Note that the functions $J_{\beta + [0,1)}$ are exactly the functions $j_{1,1,\beta}$ in \cite{projmultres1} (and indeed this inspired the notation $J_B$) and that $\sec{G}{\Lambda}$ is isomorphic to $X(1,1)$, only that the latter is considered a right module instead of a left module.
\end{example}

\section{Cutoff functions in the Feichtinger algebra}\label{appendix:s0}

In this appendix, we have collected some technical results concerning $S_0(G)$ for which the author was unable to find a reference.

The following lemma is used to prove \Cref{lem:cutoff}.

\begin{lemma}\label{lem:topgroupsets}
The following hold in a locally compact abelian group:
\begin{enumerate}
\item\label{it:top1} If $K \subseteq U$ where $K$ is compact and $U$ is open, then there exists a neighbourhood $N$ of the identity such that $KN \subseteq U$.
\item\label{it:top2} If $K \subseteq V \subseteq \overline{V} \subseteq U$ where $K$ is compact and $U,V$ are open and $U$ is contained in compact set, then there exists a neighbourhood $N$ of the identity such that $K \subseteq \cap_{x \in N}(xV)$ and $\overline{V}N \subseteq U$.
\end{enumerate}
\end{lemma}

\begin{proof}
We begin by proving \ref{it:top1}. For every $x \in K$, we can find a neighbourhood $N_x$ of the identity such that $x \in x N_x^2 \subseteq U$ \cite[p.\ 18, Theorem (4.5)]{hewittross}. Then we have that the union $\bigcup_{x \in K} (x N_x)$ covers $K$, so there exists a finite subcovering given by $x_1, \ldots, x_n \in K$. Let $N = \bigcap_{k=1}^n N_{x_k}$, which is a nonempty neighbourhood of the identity. If $x \in K$ and $y \in N$, then $x = x_k z$ for some $k$ and $z \in N_{x_k}$. But then $zy \in N_{x_k}^2$, so $xy = x_k zy \in x_k N_{x_k}^2 \subseteq U$. This shows that $KN \subseteq U$.

We now prove \ref{it:top2}. Denote by $L$ the compact set that covers $U$. Since $\overline{V}$ is contained in $L$, it is compact. Using \ref{it:top1}, we can find an open neighbourhood $N_1$ of the identity such that $\overline{V}N_1 \subseteq U$. Taking relative complements in $L$, we obtain $L \setminus V \subseteq L \setminus K$. Now $L \setminus V$ is closed and contained in $L$, hence compact, and $L \setminus K$ is open. Thus we can apply \ref{it:top1} again and obtain an open neighbourhood $N_2$ of the identity such that $(L \setminus V)N_2 \subseteq L \setminus K$. But then
\[ L \setminus \left( (L \setminus V)N_2 \right) = L \setminus \left( \bigcup_{x \in N_2} x  (L \setminus V ) \right) = \bigcap_{x \in N_2} L \setminus (x (L \setminus V)) = \bigcap_{x \in N_2} x V \]
so $K \subseteq \bigcap_{x \in N_2} x V$. Setting $N = N_1 \cap N_2$, we see that this set satisfies both $\overline{V}N \subseteq U$ and $K \subseteq \bigcap_{x \in N} (xV)$, so the proof is finished.
\end{proof}

\begin{proposition}\label{lem:cutoff}
Let $K$ be a compact subset of $G$, and let $U$ be an open subset of $G$ such that $K \subseteq U$. Then there exists a function $f \in S_0(G)$ such that $f|_K = 1$ and $\supp(f) \subseteq U$.
\end{proposition}

\begin{proof}
If $K = \emptyset$ then one can just set $f = 0$, so suppose that $K \neq \emptyset$. Since all locally compact groups are normal as topological spaces \cite[p.\ 76, Theorem (8.13)]{hewittross}, we can find an open set $V$ such that $K \subseteq V \subseteq \overline{V} \subseteq U$, and since $K$ is compact, we can assume that $\overline{V}$ compact. Now without loss of generality, we can assume that $U$ is contained in a compact set.

Let $N$ be an open neighbourhood as in part \ref{it:top2} of \Cref{lem:topgroupsets}, i.e. we have both $K \subseteq \bigcap_{t \in N} t V$ and $N\overline{V} \subseteq U$. Since $K \neq \emptyset$, $N$ is necessarily nonempty. Let $\phi \in C_c(G)$ satisfy $\int_G \phi(x) \, \dee x = 1$ and $\supp(\phi) \subseteq N$. Define $f = \phi * \chi_V$. Then $f \in S_0(G)$ by \cite[Lemma 4.2 (iii)]{notnew}. Now since $\supp(\phi) \subseteq N$, we have that
\[ f(x) = \int_G \phi(t) \chi_V(t^{-1}x) \, \dee t = \int_{N} \phi(t) \chi_V(t^{-1}x) \, \dee t .\]
If $x \in K$ then $x \in \bigcap_{t \in N} tV$, so $t^{-1}x \in V$ for all $t \in N$. Thus, $\chi_V(t^{-1}x) = 1 $ for all $t \in N$ and hence
\[ f(x) = \int_{N} \phi(t) \, \dee t = 1 .\]
If $x \notin U$ then $x \notin \overline{V}N$ which means $t^{-1}x \notin \overline{V}$ for all $t \in N$, so $\chi_V(t^{-1} x) = 0$. This gives $f(x) = 0$. While this does not necessarily give $x \notin \supp(f)$, we can pick another open set $W$ of $G$ with $K \subseteq W \subseteq \overline{W} \subseteq U$ and choose an $f \in S_0(G)$ as above with $f|_K = 1$ and $f(x) = 0$ whenever $x \notin W$. In that case, $\supp(f) \subseteq \overline{W} \subseteq U$, which gives us what we wanted.
\end{proof}

The proof of the next result is based on the proof given in \cite[Lemma 8.2.1 c)]{groechenig}. The proof employs the Wiener algebra $W(G)$ \cite{wieneralg}. Specifically, we will use that $S_0(G) \subseteq W(G)$, which can be seen by combining $W(\mathcal{F}L^1,L^1) = S_0(G)$, cf.\ \cite[Remark 6]{new_segal}, with the inclusion $W(\mathcal{F}L^1,L^1) \subseteq W(L^\infty,L^1)$, cf.\ \cite[Lemma 1.2(iv)]{feicht_interp} and $W(G) = W(L^\infty,L^1)$.

\begin{proposition}\label{prop:feichtinger_zak_cont_app}
Let $\Lambda$ be a lattice in a second countable, locally compact abelian group. If $\xi \in S_0(G)$, then $Z_{G,\Lambda}\xi$ is continuous. Hence, the Zak transform maps $S_0(G)$ into $\sec{G}{\Lambda}$.
\end{proposition}

\begin{proof}
Suppose $\xi \in S_0(G)$, so that $\xi \in W(G)$. Let $B$ be a fundamental domain for $\Lambda$ in $G$ with nonempty interior as in \Cref{prop:fund_int}. Because of quasiperiodicity, it is enough to prove continuity of $Z \xi$ on $B \times \widehat{G}$, so let $(x,\omega) \in B \times \widehat{G}$ and suppose $(x_\alpha, \omega_\alpha)_\alpha$ is a net in $B \times \widehat{G}$ that converges towards $(x,\omega)$. Let $\epsilon >0$. There exists an $\alpha_1$ such that whenever $\alpha \geq \alpha_1$, then $x_\alpha \in B^{\circ} \subseteq B$.

Since $\| \xi \|_{W(G)}$ is finite, there exists a finite subset $F$ of $\Lambda$ such that
\[ \sum_{\lambda \in \Lambda \setminus F} \sup_{x \in B} |\xi(x \lambda)| < \frac{\epsilon}{2} .\]
Since $\xi$ and $\omega$ are continuous, the function $(x,\omega) \mapsto \sum_{\lambda \in F} \xi(x \lambda) \omega(\lambda)$ is continuous, so there exists an $\alpha_2$ such that we have that
\[ \left| \sum_{\lambda \in F} \xi(x_\alpha \lambda) \omega_\alpha(\lambda) - \sum_{\lambda \in F} \xi(x \lambda) \omega(\lambda) \right| < \frac{\epsilon}{2} \]
for $\alpha \geq \alpha_2$. Combining all our observations, we see that if $\alpha \geq \alpha_1$ and $\alpha \geq \alpha_2$ then
\begin{align*}
| Z \xi(x_\alpha,\omega_\alpha) - Z \xi(x,\omega) | &= \left| \sum_{\lambda \in \Lambda} \xi(x_\alpha \lambda) \omega_\alpha(\lambda) - \sum_{\lambda \in \Lambda} \xi(x\lambda) \omega(\lambda) \right| \\
&\leq \frac{\epsilon}{2} + \left| \sum_{\lambda \in \Lambda \setminus F} \xi(x_\alpha \lambda) \omega_\alpha(\lambda) - \sum_{\lambda \in \Lambda \setminus F} \omega(x\lambda) \omega(\lambda) \right| \\
&\leq \frac{\epsilon}{2} + \sum_{\lambda \in \Lambda \setminus F} | \xi(x_\alpha \lambda)| + \sum_{\lambda \in \Lambda \setminus F} | \xi(x \lambda)| \\
&\leq \frac{\epsilon}{2} + 2 \sum_{\lambda \in \Lambda \setminus F} \sup_{x \in B} |\xi(x\lambda)| \\
&\leq \frac{\epsilon}{2} + 2 \cdot \frac{\epsilon}{4} \\
&= \epsilon.
\end{align*}
This finishes the proof.
\end{proof}

\section*{Acknowledgements}

The author is indebted to Nadia Larsen, Franz Luef, Judith Packer and Leonard Huang for valuable conversations. The author is grateful for the hospitality of Judith Packer, Carla Farsi and the rest of the operator algebras group at the University of Colorado, Boulder during his stay there in 2018--2019, during which a lot of the paper was written. The stay was supported by the Fulbright Foundation and the Tron Mohn foundation. The author also wants to thank Mads S.\ Jakobsen for pointing out helpful references, and Andreas Andersson, Are Austad, Martin Helsø and Petter Nyland for giving feedback on drafts of the paper. Finally, the author wants to thank the referee for valuable feedback, in particular for pointing out how to better structure \Cref{sec:bundles_qp}, leading to shorter proofs and less dependence on fundamental domains.

\printbibliography

\end{document}